\pdfoutput=1
\documentclass[12pt,reqno]{amsart}
\usepackage{latexsym}
\usepackage{amsmath}
\usepackage{amsthm}
\usepackage{amssymb}
\usepackage{graphicx}
\usepackage{amsfonts}
\usepackage{fullpage}
\usepackage{comment}
\usepackage{hyperref}

\title[Quantitative uniqueness estimates for $p$-Laplace type equations]{Quantitative uniqueness estimates for $p$-Laplace type equations in the plane}

\author[Chang-Yu Guo]{Chang-Yu Guo}
\address[Chang-Yu Guo]{Department of Mathematics and Statistics, University of Jyv\"askyl\"a and Department of Mathematics, University of Fribourg}
\email{changyu.guo@unifr.ch}

\author[Manas Kar]{Manas Kar}
\address[Manas Kar]{Department of Mathematics and Statistics, University of Jyv\"askyl\"a}
\email{manas.m.kar@jyu.fi}

\newcommand*{\R}{\mathbb{R}}

\newcommand*{\doo}{\partial}
\newcommand*{\ol}[1]{\overline{#1}}

\newcommand{\abs}[1]{\left| #1 \right|}

\theoremstyle{plain}
\newtheorem{theorem}{Theorem}
\newtheorem{lemma}[theorem]{Lemma}

\newtheorem{proposition}[theorem]{Proposition}

\theoremstyle{definition}

\theoremstyle{remark}
\newtheorem{remark}[theorem]{Remark}

\newtheorem{case}{Case}

\DeclareMathOperator{\dive}{div}

\def\vint_#1{\mathchoice%
          {\mathop{\kern 0.2em\vrule width 0.6em height 0.69678ex
depth -0.58065ex
                  \kern -0.8em \intop}\nolimits_{\kern -0.4em#1}}%
          {\mathop{\kern 0.1em\vrule width 0.5em height 0.69678ex
depth -0.60387ex
                  \kern -0.6em \intop}\nolimits_{#1}}%
          {\mathop{\kern 0.1em\vrule width 0.5em height 0.69678ex
              depth -0.60387ex
                  \kern -0.6em \intop}\nolimits_{#1}}%
          {\mathop{\kern 0.1em\vrule width 0.5em height 0.69678ex
depth -0.60387ex
                  \kern -0.6em \intop}\nolimits_{#1}}}
\def\vintslides_#1{\mathchoice%
          {\mathop{\kern 0.1em\vrule width 0.5em height 0.697ex depth -0.581ex
                  \kern -0.6em \intop}\nolimits_{\kern -0.4em#1}}%
          {\mathop{\kern 0.1em\vrule width 0.3em height 0.697ex depth -0.604ex
                  \kern -0.4em \intop}\nolimits_{#1}}%
          {\mathop{\kern 0.1em\vrule width 0.3em height 0.697ex depth -0.604ex
                  \kern -0.4em \intop}\nolimits_{#1}}%
          {\mathop{\kern 0.1em\vrule width 0.3em height 0.697ex depth -0.604ex
                  \kern -0.4em \intop}\nolimits_{#1}}}

\newcommand{\aveint}[2]{\mathchoice%
          {\mathop{\kern 0.2em\vrule width 0.6em height 0.69678ex
depth -0.58065ex
                  \kern -0.8em \intop}\nolimits_{\kern -0.45em#1}^{#2}}%
          {\mathop{\kern 0.1em\vrule width 0.5em height 0.69678ex
depth -0.60387ex
                  \kern -0.6em \intop}\nolimits_{#1}^{#2}}%
          {\mathop{\kern 0.1em\vrule width 0.5em height 0.69678ex
depth -0.60387ex
                  \kern -0.6em \intop}\nolimits_{#1}^{#2}}%
          {\mathop{\kern 0.1em\vrule width 0.5em height 0.69678ex
depth -0.60387ex
                  \kern -0.6em \intop}\nolimits_{#1}^{#2}}}

\numberwithin{theorem}{section}
\numberwithin{equation}{section}

\begin{document}

\begin{abstract}
In this article our main concern is to prove the quantitative unique estimates for the $p$-Laplace equation, $1<p<\infty$, with a locally Lipschitz drift in the plane. To be more precise, let $u\in W^{1,p}_{loc}(\R^2)$ be a nontrivial weak solution to
\[
\text{div}(\abs{\nabla u}^{p-2} \nabla u) + W\cdot(\abs{\nabla u}^{p-2}\nabla u) = 0 \ \text{ in }\ \mathbb{R}^2,
\]
where $W$ is a locally Lipschitz real vector satisfying $\|W\|_{L^q(\R^2)}\leq \tilde{M}$ for $q\geq \max\{p,2\}$. Assume that $u$ satisfies certain a priori assumption at 0. For $q>\max\{p,2\}$ or $q=p>2$, if $\|u\|_{L^\infty(\R^2)}\leq C_0$, then $u$ satisfies the following asymptotic estimates at $R\gg 1$
\[
\inf_{\abs{z_0}=R}\sup_{\abs{z-z_0}<1} \abs{u(z)} \geq e^{-CR^{1-\frac{2}{q}}\log R},
\]
where $C>0$ depends only on $p$, $q$, $\tilde{M}$ and $C_0$. When $q=\max\{p,2\}$ and $p\in (1,2]$, if $\abs{u(z)}\leq \abs{z}^m$ for $\abs{z} >1$ with some $m>0$, then we have
\[
\inf_{\abs{z_0}=R} \sup_{\abs{z-z_0}<1} \abs{u(z)} \geq C_1 e^{-C_2(\log R)^2},
\]
where $C_1>0$ depends only on $m, p$ and $C_2>0$ depends on $m, p, \tilde{M}$. As an immediate consequence, we obtain the strong unique continuation principle (SUCP) for nontrivial solutions of this equation. We also prove the SUCP for the weighted $p$-Laplace equation with a locally positive locally Lipschitz weight.
\end{abstract}

\maketitle

\section{Introduction}\label{sec:introduction}

In this paper, we study the unique continuation principle (UCP) for certain nonlinear elliptic partial differential equations in the plane. This principle, which states that any solution of an elliptic equation that vanishes in a small ball must be identically zero, is a fundamental property that has various applications e.g. in solvability questions, inverse problems, and control theory. 

This problem has significant differences for linear and nonlinear equations. For linear elliptic PDEs, Garofalo and Lin~\cite{Garofalo:Lin1986,Garofalo:Lin1987} have obtained even the strong unique continuation principle (SUCP) for solutions in all dimensions. As for nonlinear elliptic PDEs, very little is known. As a matter of fact, it remains unknown even for the $p$-Laplace equations in higher dimensions; see~\cite{Armstrong:Silvestre2011,Granlund:Marola2014} and the references therein for partial positive results in higher dimensions. The difficulty in establishing such a unique continuation result for nonlinear equations (such as the $p$-Laplace equation) lies in the fact that the best a priori regularity available for (viscosity or weak) solutions of the PDE is $C^{1,\alpha}$ and consequently, one can not linearize the corresponding nonlinear equation and apply the known unique continuation results for linear equations. 

When restricted to the planar case, the situation is slightly different. For instance, the UCP holds for solutions of the $p$-Laplace equations or even more general nonlinear elliptic equations in the plane; see~\cite{Bojarski:Iwaniec:1987,Granlund:Marola2012,Heinonen:Kilpelainen:Martio:1993}. The main reason for this is that planar nonlinear elliptic equations are naturally tied with the complex function theory. In the linear equation case, the real part and the complex part of an analytic function is harmonic and each harmonic function (locally) raises as the real part of an analytic function. In the nonlinear $p$-Laplace equation case ($1<p<\infty$), as discovered by Bojarski and Iwaniec~\cite{Bojarski:Iwaniec:1987}, the complex gradient of a $p$-harmonic function is a quasiregular mapping. The nice properties of quasiregular mappings in return yield many surprising results for planar $p$-harmonic functions, in particular, the UCP holds for these functions.

We will mainly consider the SUCP for two variants of the $p$-Laplace equations in this paper. The first one is the following $p$-Laplace equation, $1<p<\infty$, with a locally Lipschitz drift term
\begin{equation}\label{eq:main equation}
	\text{div}(\abs{\nabla u}^{p-2} \nabla u) + W\cdot(\abs{\nabla u}^{p-2}\nabla u) = 0 \ \text{ in } \mathbb{R}^2,
\end{equation}
where $W= (W_1, W_2)$ is a real vector-valued locally Lipschitz function with bounded $L^q$-norm for $\max\{2,p\}\leq q<\infty$. As we will see in a moment, the SUCP for solutions of \eqref{eq:main equation} will be a simple consequence of certain lower estimate of the decay rate under certain a priori assumptions for the solutions. The decay estimate will also yield quantitative uniqueness estimates of the solutions at large scale.

This kind of quantitative decay estimate problem was originally posed by Landis in the 60's \cite{Land1988}. He conjectured that if $u$ is a bounded solution of
\begin{equation}\label{eq:landis}
	\dive(\nabla u)+Vu=0\ \text{in}\ \R^n,
\end{equation}
with $\|V\|_{L^\infty(\R^n)}\leq 1$ and $|u(x)|\leq Ce^{-C|x|^{1+}}$, then is identically zero. This conjecture was disproved by Meshkov~\cite{M1991}, who constructed a $V(x)$ and a nontrivial $u(x)$ with $|u(x)|\leq C e^{-C|x|^{\frac{4}{3}}}$ satisfying \eqref{eq:landis}. He also showed that if $|u(x)|\leq C e^{-C|x|^{\frac{4}{3}+}}$, then $u\equiv 0$ in $\R^n$. A quantitative form of Meshkov's result was derived by Bourgain and Kenig \cite{Bourgain:Kenig:2005} in their resolution of Anderson localization for the Bernoulli model in higher dimensions. It should be noted that both $V$ and $u$ constructed by Meshkov are complex-valued functions; see also~\cite{D2014,Lin:Wang:2014} for positive quantitative results along this direction. 

In a recent paper of Kenig--Silvestre--Wang \cite{Kenig:Silvestre:Wang:2015}, the authors studied Landis' conjecture for second order elliptic equations in the plane in the real setting, including a special case of~\eqref{eq:main equation}. In particular, it was proved there that if $u\in W^{2,q}_{loc}(\R^2)$ is a real-valued solution of \eqref{eq:main equation} with $p=2$, $q=\infty$ such that $|u(z)|\leq e^{C_0\abs{z}}$, $|\nabla u(0)|=1$, and $\|W\|_{L^\infty(\R^2)}\leq 1$, then
\begin{align}\label{eq:KSW main result}
	\inf_{|z_0|=R}\sup_{|z-z_0|<1}|u(z)|\geq e^{-CR\log R}\ \quad \text{for } R\gg 1,
\end{align}
where the constant $C$ depends only on $C_0$. Later, Kenig and Wang \cite{Kenig:Wang:2015} obtained similar estimates as~\eqref{eq:KSW main result} for the equation~\ref{eq:main equation} with $p=2$ under an $L^q$-boundedness assumption for $W$ for $2\leq q<\infty$. To be more precise, let $u\in W^{2,q}_{loc}(\R^2)$ be a real solution of~\eqref{eq:main equation} with $p=2$. Suppose $2<q<\infty$. If $|u(z)|\leq C_0$, $|\nabla u(0)|=1$, $\|W\|_{L^q(\R^2)}\leq \tilde{M}$, then 
\begin{align}\label{eq:KW main result}
\inf_{|z_0|=R}\sup_{|z-z_0|<1}|u(z)|\geq e^{-CR^{1-\frac{2}{q}}\log R}\ \quad \text{for } R\gg 1,
\end{align}
where $C$ depends only on $q$, $\tilde{M}$ and $C_0$. Similar quantitative estimates (but of different order) hold for the case $q=2$ as well.

Our first main result is the following quantitative lower bound decay for nontrivial solutions of~\eqref{eq:main equation} that generalizes the above-mentioned result of Kenig and Wang~\cite{Kenig:Wang:2015}.
\begin{theorem}\label{thm:main thm}
Let $u\in W_{loc}^{1,p}(\mathbb{R}^2)$, $1<p<\infty$, be a weak solution of \eqref{eq:main equation} with $W$ being locally Lipschitz. 
\begin{description}
	\item[i). $q > \max \{2,p\}$] If $\|u\|_{L^\infty(\R^2)}\leq C_0$, $\abs{\nabla u(0)} = 1$, and $\|W\|_{L^q(\mathbb{R}^2)}\leq \tilde{M}$,
	then
	\begin{align}\label{eq:main estimate}
		\inf_{\abs{z_0}=R}\sup_{\abs{z-z_0}<1} \abs{u(z)} \geq e^{-CR^{1-\frac{2}{q}}\log R}
	\end{align}
	for $R\gg 1$, where $C$ depends only on $p$, $q$, $\tilde{M}$ and $C_0$.
	
	\item[ii). $q=\max\{2,p\}$ and $p>2$]If $\|u\|_{L^\infty(\R^2)}\leq C_0$, $\abs{\nabla u(0)} = 1$, and $\|W\|_{L^p(\mathbb{R}^2)}\leq \tilde{M}$,
	then
	\begin{align}\label{eq:main estimate 3}
	\inf_{\abs{z_0}=R}\sup_{\abs{z-z_0}<1} \abs{u(z)} \geq e^{-CR^{1-\frac{2}{p}}\log R}
	\end{align}
	for $R\gg 1$, where $C$ depends only on $p$, $\tilde{M}$ and $C_0$.
	
	In addition, if we assume $\abs{u(z)}\leq \abs{z}^m$ for $\abs{z}>1$, with some $m>0$, then the similar estimates as \eqref{eq:main estimate} and \eqref{eq:main estimate 3} are also true.
	\item[iii). $q=\max\{2,p\}$ and $1<p\leq 2$]If we assume $\|u\|_{L^\infty(\R^2)}\leq C_0$, $\|\nabla u\|_{L^p(B_1)}\geq 1$, and  $\|W\|_{L^2(\mathbb{R}^2)} \leq \ \tilde{M}$, then
	\begin{equation}\label{eq:main estimate 2}
	\inf_{\abs{z_0}=R} \sup_{\abs{z-z_0}<1} \abs{u(z)} \geq R^{-C},
	\end{equation}
	for $R\gg1$, where $C>0$ depends on $p$, $\tilde{M}$ and $C_0$.
	
	Moreover, if $\abs{u(z)}\leq \abs{z}^m$ for $\abs{z}>1$, with some $m>0$, $\|\nabla u\|_{L^p(B_1)}\geq 1$, and  $\|W\|_{L^2(\mathbb{R}^2)} \leq \ \tilde{M}$, then
\begin{equation}\label{eq:main estimate 4}
\inf_{\abs{z_0}=R} \sup_{\abs{z-z_0}<1} \abs{u(z)} \geq C_1 e^{-C_2(\log R)^2},
\end{equation}
where $C_1>0$ depends only on $m, p$ and $C_2>0$ depends on $m, p, \tilde{M}$.
\end{description}

\end{theorem}

Note that since we are dealing with the nonlinear equation~\eqref{eq:main equation}, the a priori assumption $u\in W^{2,q}_{loc}(\R^2)$ from~\cite{Kenig:Wang:2015} has to be replaced with the substantially weaker assumption $u\in W^{1,q}_{loc}(\R^2)$. For technical reasons, we have also assumed that the drift $W$ is locally Lipschitz and we believe that it is sufficient to assume that $W$ is locally bounded; see the discussion in Section~\ref{sec:concluding remark}.

In the formulation of \textbf{i)} and \textbf{ii).} in Theorem~\ref{thm:main thm}, we have an a priori global $L^\infty$-boundedness assumption on the solution $u$. This assumption can be weakened to the pointwise boundedness assumption as in \textbf{iii)} of Theorem~\ref{thm:main thm}: see Theorem~\ref{mainpq2} below. As a matter of fact, upon changing the form of the lower bound in Theorem~\ref{thm:main thm}, one can impose any a priori upper growth condition on the solution.

As in \cite{Kenig:Wang:2015}, a key reduction we will use in this paper is the scaling argument introduced by Bourgain and Kenig \cite{Bourgain:Kenig:2005}. It allows us to reduce our problem to certain estimate of the maximal vanishing order of the solution $v$ to the equation
\begin{equation}\label{eq:reduced equation}
\dive(\abs{\nabla v}^{p-2} \nabla v) + A\cdot(\abs{\nabla v}^{p-2}\nabla v) = 0 \ \text{ in } B_8,
\end{equation}
with $A= (A_1, A_2)$ being Lipschitz and $\|A\|_{L^q(B_8)} \leq M$. The proof of the maximal vanishing order of $v$ relies on a nice reduction of \eqref{eq:reduced equation} to a quasilinear Beltrami equation. Using the well-known theory of Beltrami operators, we may represent the solutions of the Beltrami equation via quasiregular mappings (cf.~Theorem \ref{ThpLLL}). Unlike the linear case $p=2$ in \cite{Kenig:Wang:2015},  we need a version of the Hadamard's three circle theorem for quasiregular mappings (cf.~Proposition \ref{prop:key prop}) to derive the vanishing order. The
case $q=\max\{2,p\}$ needs special attention due to the fact that the Cauchy transform
fails to be a bounded mapping from $L^2$ to $L^\infty$.

As in the linear case \cite{Kenig:Wang:2015}, a quantitative form of the estimate of the maximal vanishing order of $v$ provides us the strong unique continuation property (SUCP) for \eqref{eq:reduced equation}. 

\begin{theorem}\label{thm:SUCP}
	Let $\Omega\subset \R^2$ be a domain. Let $v\in W_{loc}^{1,p}(\Omega), 1<p<\infty,$ be a weak solution of
	\[
	\dive(\abs{\nabla v}^{p-2}\nabla v) + A\cdot(\abs{\nabla v}^{p-2}\nabla v) = 0 \ \text{in}\ \Omega,
	\]
	where $A$ is locally Lipschitz in $\Omega$. If for some $z_0\in \Omega$ and for all $N\in \mathbb{N}$, there exist $C_N>0$ and $r_N>0$ such that
	\[
	\abs{v(z)-v(z_0)} \leq C_N\abs{z-z_0}^N, \ \forall \ \abs{z-z_0}<r_N,
	\]
	then
	\[
	v(z) \equiv v(z_0).
	\]
	
\end{theorem}

The second class of equations we are considering in this paper is the following weighted $p$-Laplace equation in a planar domain $\Omega$ for $1<p<\infty$:
\begin{equation}\label{eq:weighted p laplace}
	\dive(\sigma \abs{\nabla u}^{p-2} \nabla u) = 0 \ \text{ in } \Omega,
\end{equation}
where $\sigma\in {W_{loc}^{1,\infty}(\Omega)}$ is a locally positive locally Lipschitz continuous function. Here by saying that $\sigma$ is locally positive\footnote{This terminology is not standard. The term ``locally positive" is mainly used to distinguish the term ``positive" in inverse problems, which usually means a uniform positive lower bound (cf.~\cite{Guo:Kar:Salo:2015}).}, we mean that for each $K\subset\subset \Omega$,  there exists a positive constant $c_K$ such that $\sigma>c_K$ in $K$. The main result is the following SUCP for solutions of~\eqref{eq:weighted p laplace}. Note that the UCP for solutions of~\eqref{eq:weighted p laplace} are well-known; see e.g.~\cite{Alessandrini:1987,Manfredi:1988} for the constant conductivity $\sigma$, and \cite{Granlund:Marola2012,Guo:Kar:Salo:2015} for more general case.

\begin{theorem}\label{thm:SUCP weighted p-la}
Let $\Omega\subset\mathbb{R}^2$ be a domain and $\sigma\in W_{loc}^{1,\infty}(\Omega)$ a locally positive locally Lipschitz weight. Let $v\in W_{loc}^{1,p}(\Omega)$ be a weak solution of the weighted $p$-Laplace equation \eqref{eq:weighted p laplace} with  $1<p<\infty$. If for some $z_0\in \Omega$ and for all $N\in \mathbb{N}$, there exist $C_N>0$ and $r_N>0$ such that
\[
\abs{v(z)-v(z_0)} \leq C_N\abs{z-z_0}^N, \ \forall \ \abs{z-z_0}<r_N,
\]
then
\[
v(z) \equiv v(z_0).
\]
\end{theorem}

One can also prove a version of Theorem~\ref{thm:main thm} for the weighted $p$-Laplace equation~\eqref{eq:weighted p laplace} and deduce Theorem~\ref{thm:SUCP weighted p-la} as an immediate consequence of that: see Remark~\ref{rmk:on QUCP for p-Laplace type equations} below. As our interest lies in the SUCP for equation~\eqref{eq:weighted p laplace}, we do not formulate our theorem in that fashion. 

The main idea of the proof relies on similar localizing arguments as in the previous case. Namely, it suffices to consider our problem in the disk $B_8$.
We will show that the estimate of the maximal vanishing order of the solution $v$, Theorem \ref{thm:maximal vanishing order}, to the equation
\begin{equation}\label{eq11}
\dive(A \abs{\nabla v}^{p-2} \nabla v) = 0 \ \text{ in } \ B_8,
\end{equation}
with $A\in W_{loc}^{1,\infty}(B_8)$, being a positive Lipschitz coefficient, is sufficient to establish Theorem \ref{thm:SUCP weighted p-la}. As in the previous case ($p$-Laplace with a locally Lipschitz drift), we will reduce the original equation~\eqref{eq11} to certain quasilinear Beltrami equation \eqref{eq:for Fa}. Then the derivation of the estimate of the maximal vanishing order for $v$ will follow from the explicit representation of solutions of the Beltrami equation, the appropriate use of Caccioppoli's inequality for the weighted $p$-Laplace equation, and the Hadamard's three circles theorem for quasiregular mappings. When $1<p\leq 2$, the local Lipschitz regularity assumption on $\sigma$ can be weakened and we discuss this improvement in Section \ref{subsec:weaken the regularity}.

This paper is organized as follows. Section~\ref{sec:Partial regularity of weak solutions} contains the basic definition of weak solutions to~\eqref{eq:main equation} and some partial regularity result of the gradient of solutions for the equation~\eqref{eq:main equation}. In Section \ref{sec:q>p 2}, we consider the case $q>\max\{p,2\}$ for the proof of Theorem \ref{thm:main thm} and the case $q=\max\{p,2\}$ is treated in Section \ref{sec:q=p 2}. Section \ref{sec:weighted p laplace} contains the proof of Theorem~\ref{thm:SUCP weighted p-la}. The final section, Section \ref{sec:concluding remark} contains a short remark on reducing the local regularity assumption in Theorem~\ref{thm:main thm}. Throughout the paper, $C$ stands for an absolute constant whose dependence will be specified if necessary. Its value may vary from line to line.

\subsection*{Acknowledgements}

C.Y.\ Guo was supported by the Finnish Cultural Foundation--
Central Finland Regional Fund (grant number 30151735). M.\ Kar\ was partly supported by an ERC Starting Grant (grant agreement no 307023) and by the Academy of Finland through the Centre of Excellence in Inverse Problems Research. We would like to thank Prof.~Tero Kilpel\"ainen, Prof.~Kari Astala, Prof.~Xiao Zhong and Dr.~Chang-Lin Xiang for their helpful suggestions.

\section{Partial regularity of weak solutions}\label{sec:Partial regularity of weak solutions}
Throughout this paper, $\Omega\subset\mathbb{R}^2$ will be a domain, i.e., an open connected set in $\R^2$. A real function $u\in W^{1,p}_{loc}(\Omega)$ is said to be a weak solution of~\eqref{eq:main equation} if 
\begin{align}\label{eq:def for weak solution}
 \int_{\Omega}|\nabla u|^{p-2}\nabla u\cdot \nabla \eta dx=\int_{\Omega}W\cdot(|\nabla u|^{p-2}\nabla u)\eta dx	
\end{align}
for all $\eta\in C_0^\infty(\Omega)$. Note that $W\in L^\infty_{loc}(\Omega)$ and $u\in W^{1,p}_{loc}(\Omega)$, the right-hand side of~\eqref{eq:def for weak solution} is integrable. (As a matter of fact, since $|\nabla u|^{p-2}\nabla u\in L^{p/(p-1)}_{loc}(\Omega)$, it suffices to assume $W\in L^s_{loc}(\Omega)$ for $s$ bigger than the dual exponent of $p/(p-1)$). Moreover, it is clear from our assumption that if $u\in W^{1,p}_{loc}(\Omega)$ is a weak solution of \eqref{eq:def for weak solution}, then \eqref{eq:def for weak solution} holds for all $\eta\in W^{1,p}_0(\Omega)$. 

If $u\in W^{1,p}_{loc}(B_8)$ is a weak solution of~\eqref{eq:def for weak solution}, we consider the mapping $F=|\nabla u|^{(p-2)/2}\nabla u$. It is clear that $F\in L^2_{loc}(B_8)$. When $W$ is locally Lipschitz, the next lemma implies that $F$ enjoys higher regularity. The proof is similar to that used in~\cite[Theorem 16.3.1]{Astala:Iwaniec:Martin}. 
\begin{lemma}\label{lemma:regularity of F}
Assume that $W$ is locally Lipschitz in $B_8$. Then	$F\in W^{1,2}_{loc}(B_8,\R^2)$.
\end{lemma}
\begin{proof}
	Since most of the argument we are adapting here is similar to~\cite[Proof of Theorem 16.3.1]{Astala:Iwaniec:Martin}, see also~\cite[Proposition 2]{Bojarski:Iwaniec:1987}, we will only outline the main steps and differences, and refer the interested readers to \cite{Astala:Iwaniec:Martin,Bojarski:Iwaniec:1987} for the omitted details.
	
	For a compact set $K\subset B_8$ and choose $h\in \R^2$ such that $|h|<d(K,\partial B_8)$. Note that both functions
	\begin{align*}
		\phi(z)=\eta^2(z)\big(u(z+h)-u(z)\big)\quad \text{and}\quad \phi(z-h)
	\end{align*}
	belong to $W^{1,p}_0(B_8)$. Thus it is legitimate to write the identities
	\begin{align*}
		\int_{B_8}|\nabla u(z)|^{p-2}\nabla u(z)\cdot \nabla \phi(z)dz=\int_{B_8}|\nabla u(z)|^{p-2}W(z)\cdot \nabla u(z)\phi(z)dz
	\end{align*}
	and
	\begin{align*}
	\int_{B_8}|\nabla u(z+h)|^{p-2}\nabla u(z+h)\cdot \nabla \phi(z)dz=\int_{B_8}|\nabla u(z+h)|^{p-2}W(z+h)\cdot \nabla u(z+h)\phi(z)dz.
	\end{align*}
	It follows that
	\begin{align*}
		\int_{B_8}&\langle |\nabla u(z+h)|^{p-2}\nabla u(z+h)-|\nabla u(z)|^{p-2}\nabla u(z), \nabla \phi(z)\rangle dz\tag{I}\\
		&=\int_{B_8} \big(|\nabla u(z+h)|^{p-2}W(z+h)\cdot \nabla u(z+h)-|\nabla u(z)|^{p-2}W(z)\cdot\nabla u(z)\big)\phi(z)dz.
	\end{align*}
	Substituting the form
	\begin{align*}
		\nabla \phi(z)=2\eta(z)(u(z+h)-u(z))\nabla \eta(z)+\eta(z)^2(\nabla u(z+h)-\nabla u(z))
	\end{align*}
	in (I) and using some elementary inequalities (cf.~\cite[Proof of Theorem 16.3.1]{Astala:Iwaniec:Martin}), we conclude that
	\begin{align*}
		\int_{B_8}&\eta(z)^2|F(z+h)-F(z)|^2dz\\
		&\leq 2\Big|\int_{B_8}\eta(z)[u(z+h)-u(z)]\langle |\nabla u(z+h)|^{p-2}\nabla u(z+h)-|\nabla u(z)|^{p-2}\nabla u(z), \nabla \eta(z)\rangle dz\Big|\\
		&\qquad+ \Big|\int_{B_8} \big(|\nabla u(z+h)|^{p-2}W(z+h)\cdot \nabla u(z+h)-|\nabla u(z)|^{p-2}W(z)\cdot\nabla u(z)\big)\phi(z)dz\Big|\tag{II}.
	\end{align*}
	By the proof of Theorem 16.3.1, the first term in (II) can be bounded from above by
	\begin{align*}
		\text{T}=O(|h|)\Big(\int_{B_8}&\eta(z)^2|F(z+h)-F(z)|^2dz\Big)^{1/2}.
	\end{align*}
	Thus it suffices to show that the second term can be bounded by $O(|h|^2)$ or $T$. To this end, we write
	\begin{align*}
		\int_{B_8} \big(|\nabla u(z+h)|^{p-2}W(z+h)\cdot \nabla u(z+h)-|\nabla u(z)|^{p-2}W(z)\cdot\nabla u(z)\big)\phi(z)dz=\text{III+IV},
	\end{align*}
	where
	\begin{align*}
		\text{III}=\int_{B_8}\langle |\nabla u(z+h)|^{p-2}\nabla u(z+h)-|\nabla u(z)|^{p-2}\nabla u(z), W(z+h)\rangle \phi(z)dz
	\end{align*}
	and
		\begin{align*}
		\text{IV}=\int_{B_8}|\nabla u(z)|^{p-2}\langle W(z+h)-W(z), \nabla u(z)\rangle\phi(z)dz.
		\end{align*}
	By the elementary inequality from~\cite[Proof of Theorem 16.3.1]{Astala:Iwaniec:Martin}, $|$III$|$ can be bounded from above by
	\begin{align*}
		&\|W\|_{L^\infty(B_8)}
		\int_{B_8}\Big||\nabla u(z+h)|^{p-2}\nabla u(z+h)-|\nabla u(z)|^{p-2}\nabla u(z)\Big||\phi(z)|dz\\
		&\leq c(p)\|W\|_{L^\infty(B_8)}\int_{B_8}\Big(|\nabla u(z+h)|^p+|\nabla u(z)|^p\Big)^{(p-2)/2p}|F(z+h)-F(z)||\phi(z)|dz.
	\end{align*}
	By H\"older's inequality, we have
	\begin{align*}
		\int_{B_8}&\Big(|\nabla u(z+h)|^p+|\nabla u(z)|^p\Big)^{(p-2)/2p}|F(z+h)-F(z)||\phi(z)|dz\\
		&\leq c\Big(\int_{B_8}|u(z+h)-u(z)|^pdz\Big)^{1/p}\cdot\Big(\int_{B_8}|\nabla u(z+h)|^p+|\nabla u(z)|^pdz\Big)^{(p-2)/2p}\\
		&\qquad\qquad\qquad\cdot \Big(\int_{B_8}\eta(z)^2|F(z+h)-F(z)|^2dz\Big)^{1/2}\\
		&\leq O(|h|)\Big(\int_{B_8}\eta(z)^2|F(z+h)-F(z)|^2dz\Big)^{1/2}=T.
	\end{align*}
	It remains to estimate $|$IV$|$. Note that $W$ is locally Lipschitz. By H\"older's inequality, we have the bound
	\begin{align*}
     L|h|\Big(\int_{B_8} |\nabla u(z)|^pdz \Big)^{(p-1)/p}\Big(\int_{B_8}|u(z+h)-u(z)|^pdz\Big)^{1/p}=O(|h|^2),
	\end{align*}
	where $L$ depends on the Lipschitz constant of $W$ in $B_8$.
	In conclusion, we have shown that
	\begin{align*}
		\Big(\int_{B_8}\eta(z)^2|F(z+h)-F(z)|^2dz\Big)^{1/2}=O(|h|)
	\end{align*}
	and thus $F\in W^{1,2}_{loc}(B_8)$.
\end{proof}

\section{SUCP I: $q>\max\{p,2\}$}\label{sec:q>p 2}

Our main concern in this section is to consider the equation \eqref{eq:main equation}
with $\|W\|_{L^q(\mathbb{R}^2)} \leq \tilde{M}$ for $1<p<\infty$ and $\max\{p,2\}<q<\infty$. As pointed out in the introduction, by using the scaling argument from~\cite{Bourgain:Kenig:2005}, we only need to consider the problem in the disc $B_8$. It is convenient for us to restate the problem as follows: Let $v$ be a solution of 
\begin{equation}\label{PLapLOGT}
\text{div}(\abs{\nabla v}^{p-2} \nabla v) + A\cdot(\abs{\nabla v}^{p-2}\nabla v) = 0 \ \text{ in } B_8,
\end{equation}
with $A= (A_1, A_2)$ being Lipschitz in $B_8$, $\|A\|_{L^q(B_8)} \leq M, M\geq 1$.

We will need the planar theory of quasiconformal mappings and quasiregular mappings in a sequence of results below and we refer the readers to the excellent book~\cite{Astala:Iwaniec:Martin} for a comprehensive treatment.

\subsection{Auxiliary results}
In this section, we prepare some auxiliary results that are necessary for the proof of Theorem \ref{thm:main thm}.

Denote by $G=v_x-iv_y$ the complex gradient of $v$. Our aim is to derive a non-linear Beltrami equation for (certain function of) $G$. In order to do that we mainly follow the approach by B. Bojarski and T. Iwaniec \cite{Bojarski:Iwaniec:1987}. 

Define 
$F = \abs{G}^aG$, where $a=(p-2)/2$. By Lemma~\ref{lemma:regularity of F}, $F\in W^{1,2}(B_8)$ and so we may apply a simple computation (cf.~\cite{Bojarski:Iwaniec:1987}) to deduce that $F$ satisfies 
the following quasilinear Beltrami equation
\[
\frac{\doo F}{\doo\overline{z}} = q_1 \frac{\doo F}{\doo z} + q_2 \overline{\frac{\doo F}{\doo z}} + q_3 F,
\]
where 
\[
q_1 := -\frac{1}{2}\left(\frac{p-2-a}{p+a} + \frac{a}{a+2}\right)\frac{\overline{F}}{F},
\]
\[
q_2 := -\frac{1}{2}\left(\frac{p-2-a}{p+a} - \frac{a}{a+2}\right)\frac{F}{\overline{F}},
\]
and
\[
q_3 := -\frac{2(a+1)}{a+p}\left[A_1\left(1+ \frac{\overline{F}}{F}\right) + i A_2\left(1 - \frac{\overline{F}}{F}\right) \right].
\]

The following representation result should not surprise any expert and the proof we have adapted here is very similar to that used in~\cite[Theorem 4.3]{bo09}.
\begin{theorem}\label{ThpLLL}
	Let $F$ satisfies the following quasilinear  Beltrami equation
	\begin{equation}\label{quASi1}
	\frac{\doo F}{\doo\overline{z}} = q_1 \frac{\doo F}{\doo z} + q_2 \overline{\frac{\doo F}{\doo z}} + q_3 F\ \text{in}\ B_8,
	\end{equation}
	where
	\[
	q_1 := -\frac{1}{2}\left(\frac{p-2}{p+2} + \frac{p-2}{3p-2}\right)\frac{\overline{F}}{F},
	\]
	\[
	q_2 := -\frac{1}{2}\left(\frac{p-2}{3p-2} - \frac{p-2}{p+2}\right)\frac{F}{\overline{F}},
	\]
	and
	\[
	q_3 := -\frac{2p}{3p-2}\left[A_1\left(1+ \frac{\overline{F}}{F}\right) + i A_2\left(1 - \frac{\overline{F}}{F}\right) \right].
	\]
	Then the solution is represented by
	\[
	F(z) = h\big(\phi(z)\big)e^{\omega(z)} \ \text{in}\ B_8,
	\]
	where $\phi:B_8\to \phi(B_8)$ is a $K$-quasiconformal mapping, $h:\phi(B_8)\to \R^2$ is holomorphic,  and 
	\[
	\omega(z) = T(g(z)) = -\frac{1}{\pi}\int_{B_8}\frac{g(\xi)}{\xi-z}d\xi
	\]
	for some $g\in L^\delta(B_8)$ with $\delta\in [2,\min\{p_0,q\})$, where $T$ is the Cauchy transform of $g$. Moreover, the quasiconformality coefficient $K$ and the constant $p_0>2$ depending, explicitly, only on $p$.
\end{theorem}

\begin{proof}
	Let $F$ be the solution of the quasilinear Beltrami equation~\eqref{quASi1} and let $q_i$, $i=1,2,3$, be given as in Theorem~\ref{ThpLLL}. It is clear that if $F$ satisfies \eqref{quASi1}, then it is a solution of the following differential inequality
	\begin{equation}\label{eq:differential inequality}
		\Big|\frac{\doo F}{\doo\overline{z}}\Big| \leq  k \Big|\frac{\doo F}{\doo z}\Big| + oF\ \text{in}\ B_8,
	\end{equation}
	where 
	$$k=\|q_1\|_{L^\infty(B_8)}+\|q_2\|_{L^\infty(B_8)}\in (0,1)$$
	and $o\in L^q(B_8)$. We next express the equation \eqref{eq:differential inequality} as
	\begin{equation}\label{eq:simple equation}
	\frac{\doo F}{\doo\overline{z}}= o_1\frac{\doo F}{\doo z} + o_2 F\ \text{in}\ B_8,
	\end{equation}
	where $\|o_1\|_{L^\infty(B_8)}\leq k$ and $|o_2(z)|\leq o(z)$ for a.e. $z\in B_8$.
	
	Denote by $\hat{o}_i$, $i=1,2$, the zero extension of $o_i$, i.e., simply define $o_i\equiv 0$ outside $B_8$. Consider the integral equation
	\begin{equation}\label{integral}
	(I-\hat{o}_1S)g = \hat{o}_2\ \text{in }\R^2,
	\end{equation}
	where
	$S$ is the Beurling transform of $g$ defined as
	\[
	Sg(z) := -\frac{1}{\pi}\int_{B_8}\frac{g(\xi)}{(\xi-z)^2} d\xi.
	\]
	In view of our assumptions $(A_1, A_2) \in L^q(B_8)$, $q>2$, we have $\hat{o}_2\in L^q(\R^2)$ with compact support. In particular, $\hat{o}_2\in L^\delta(\R^2)$ for any $\delta\leq q$. Note also that the Beltrami coefficient $\hat{o}_1$ is a bounded function with compact support. By the well-known $L^\delta$-theory of Beltrami operators (cf.~\cite[Chapter 14]{Astala:Iwaniec:Martin} or~\cite[Theorem 1]{Astala:Iwaniec:Saksman:2001}), we know that for each $\delta\in (1+k,\min\{q,p_0\})$, where $p_0=1+\frac{1}{k}$, there exists a unique solution $g\in L^{\delta}(\R^2)$ that solves~\eqref{integral}. 
	Set
	\[
	\omega(z) = T(g(z)) = -\frac{1}{\pi} \int_{B_8}\frac{g(\xi)}{\xi-z} d\xi,
	\]
	where $T$ is the well-known Cauchy transform of $g$.

	We now consider the function
	\[
	f(z) = F(z)e^{-\omega(z)}.
	\]
	Using the well-known facts (cf.~\cite[Chapter 4]{Astala:Iwaniec:Martin}) that
	\[
	\frac{\doo \omega}{\doo\bar{z}} = g \ \ \text{and}\ \ \frac{\doo \omega}{\doo z} = S(g),
	\]
	we easily obtain 
	\[
	f_{\bar{z}} = \frac{\doo F}{\doo\bar{z}} e^{-\omega} - g e^{-\omega} F \ \ \text{and}\ \ f_{z} = \frac{\doo F}{\doo z} e^{-\omega} - F e^{-\omega} S(g).
	\]
	Since $g$ solves \eqref{integral}, we have
	\[
	\begin{split}
	\left[\frac{\doo f}{\doo\overline{z}} - o_1 \frac{\doo f}{\doo z} \right]e^{\omega}
	& = \frac{\doo F}{\doo\overline{z}} - F\left(g-o_1S(g)\right) - o_1 \frac{\doo F}{\doo z}\\
	&=\frac{\doo F}{\doo\overline{z}} - o_1 \frac{\doo F}{\doo z}-o_2F= 0\ \text{in}\ B_8.
	\end{split}
	\]
	Since $e^{\omega}$ is non-negative, we infer that $f$ solves the following Beltrami equation
	\begin{equation}\label{BELpp}
	\frac{\doo f}{\doo\overline{z}} - o_1 \frac{\doo f}{\doo z} = 0 \ \text{in}\ B_8.
	\end{equation}
	Note that $\|o_1\|_{L^\infty(B_8)}\leq k<1$. Applying \cite[Theorem 3.3]{Bojarski:1957} or~\cite[Corollary 5.5.4]{Astala:Iwaniec:Martin}, we obtain 
	\[
	f(z) = h(\phi(z))
	\]
	where $\phi:B_8\to \phi(B_8)$ is $K$-quasiconformal with $K=\frac{1+k}{1-k}$ (depending only on $p$) and $h:\phi(B_8)\to \R^2$ is holomorphic.
	Consequently, each solution of \eqref{quASi1} is of the form 
	\[
	F(z) = h(\phi(z))e^{\omega(z)},
	\]
	where $\phi:B_8\to \phi(B_8)$ is a $K$-quasiconformal mapping, $h:\phi(B_8)\to \R^2$ is holomorphic,  and 
	\[
	\omega(z) = T(g) = -\frac{1}{\pi}\int_{B_8}\frac{g(\xi)}{\xi-z}d\xi
	\]
	with $g\in L^\delta(B_8)$ being the solution of~\eqref{integral}.

\end{proof}

\begin{remark}\label{rmk:on the k}
	It is easy to compute that	
	\begin{equation*}
	k=
	\begin{cases}
	\frac{p-2}{p+2} & \text{if } p\geq 2 \\
	\frac{2-p}{3p-2} & \text{if } 1<p\leq 2
	\end{cases}
	\end{equation*}
	and $p_0=1+\frac{1}{k}$. It is also clear from the proof of Theorem~\ref{ThpLLL} that if $2\leq q<1+\frac{1}{k}$, then the conclusion of Theorem \ref{ThpLLL} remains valid with $\delta\in [2,q]$.
\end{remark}

The following proposition will serve as the key for our later proofs.
\begin{proposition}\label{prop:key prop}
	Let $\phi\colon B_8\to \mathbb{R}^2$ be a $K$-quasiregular mapping. Then for each $r\in (0,\frac{1}{4})$, there exists $\theta=\frac{E}{\log\frac{E'}{r}}\in (0,1)$, where $E$ and $E'$, depending only on  $K$, such that
	\begin{align*}
	\|\phi\|_{L^\infty(B_1)}\leq C\big(r^{-1}\|\phi\|_{L^2(B_{r/2})}\big)^{\theta}\|\phi\|_{L^2(B_7)}^{1-\theta}.
	\end{align*}
\end{proposition}
\begin{proof}
	Since $\phi\colon B_8\to \mathbb{R}^2$ is $K$-quasiregular, by the well-known factorization (cf.~\cite[Theorem 3.3]{Bojarski:1957} or~\cite[Theorem 3.3]{bo09} or \cite[Corollary 5.5.4]{Astala:Iwaniec:Martin}), $\phi=F\circ f$, where $F\colon f(B_8)\to \mathbb{R}^2$ is an analytic function and $f\colon B_8\to f(B_8)$ is a $K$-quasiconformal mapping such that $f(0)=0$ and $f(B_6)=B_6$. In particular, $f|_{B_6}$ is $\eta$-quasisymmetric with $\eta$ depending only on $K$ (cf.~\cite[Theorem 11.19]{heinonen01}). Then it follows (from~\cite[Theorem 11.3]{heinonen01}) that there exist positive constants $C_0$ and $\alpha\in (0,1]$, depending only on $K$ (and the dimension $n=2$) such that
	\begin{align}\label{eq:bi-Holder continuous}
	C_0^{-1}|x-y|^{1/\alpha}\leq |f(x)-f(y)|\leq C_0|x-y|^{\alpha}	
	\end{align} 
	for all $x,y\in B_6$.
	
	We next show that 
	\begin{align}\label{eq:three circle}
	\|\phi\|_{L^\infty(B_1)}\leq \|\phi\|_{L^\infty(B_{r/4})}^{\theta}\|\phi\|_{L^\infty(B_6)}^{1-\theta}.
	\end{align}
	Write $r_1=r/4$, $r_2=1$ and $r_3=6$. We may select $\tilde{r}_1$ and $\tilde{r}_2$, quantitatively, so that $f(B_{r_1})\supset B_{\tilde{r}_1}$ and that $f(B_{r_2})\subset B_{\tilde{r}_2}$. For instance, we could choose $\tilde{r}_1=\big(\frac{r_1}{C_0}\big)^{1/\alpha}$ and $\tilde{r}_2=6-\big(\frac{6-r_2}{C_0}\big)^{1/\alpha}$.
	
	Then it follows from our choice and the Hadamard'€™s three circle theorem for analytic functions that
	\begin{align*}
	\|\phi\|_{L^\infty(B_{r_2})}&=\|F\circ f\|_{L^\infty(B_{r_2})}\leq \|F\|_{L^\infty(B_{\tilde{r}_2})}\\
	&\leq \|F\|_{L^\infty(B_{\tilde{r}_1})}^\theta\|F\|_{L^\infty(B_6)}^{1-\theta}\leq \|\phi\|_{L^\infty(B_{r_1})}^\theta\|\phi\|_{L^\infty(B_6)}^{1-\theta},
	\end{align*}
	where $$0<\theta=\theta(r)=\frac{\log\frac{r_3}{\tilde{r}_2}}{\log\frac{r_3}{\tilde{r}_1}}=\frac{\log\frac{6}{6-(\frac{5}{C_0})^{1/\alpha}}}{\log\frac{6}{(\frac{r}{C_0})^{1/\alpha}}}<1.$$
	For later purpose, we will refer the above result as the Hadamard'€™s three circle theorem for quasiregular mappings.
	
	We next point out the following standard fact about quasiregular mappings from~\cite[Section 14.35]{Heinonen:Kilpelainen:Martio:1993}:
	If $f\colon \Omega\to \mathbb{R}^2$ is a $K$-quasiregular mapping and if $u\colon f(\Omega)\to \mathbb{R}$ is a harmonic function, then $v=u\circ f$ is $\mathcal{A}$-harmonic (of type 2) in $\Omega$. Here $v$ being $\mathcal{A}$-harmonic (of type 2) means that $v$ is a weak solution for
	\begin{align*}
	\dive \mathcal{A}(x,\nabla v(x))=0,
	\end{align*}
	where $\mathcal{A}\colon \mathbb{R}^2\times \mathbb{R}^2\to \mathbb{R}^2$ satisfies
	\begin{itemize}
		\item[i)] $\mathcal{A}(x,\xi)\cdot \xi\geq \alpha |\xi|^2$;
		\item[ii)] $|\mathcal{A}(x,\xi)|\leq \beta |\xi|$;
		\item[iii)] $\big(\mathcal{A}(x,\xi_1)-\mathcal{A}(x,\xi_2)\big)\cdot\big(\xi_1-\xi_2\big)>0$ for all $\xi_1,\xi_2\in \mathbb{R}^2$ with $\xi_1\neq \xi_2$;
		\item[iv)] $\mathcal{A}(x,\lambda \xi)=\lambda\mathcal{A}(x,\xi)$ whenever $\lambda\neq 0$.
	\end{itemize}
	For $i=1,2$, the coordinate function $x_i$ is a harmonic function and so the component function $\phi_i$ of the $K$-quasiregular mapping $\phi$ is $\mathcal{A}$-harmonic (of type 2). Since $\phi_i$ is $\mathcal{A}$-harmonic (of type 2), standard interior estimates (see for instance~\cite[Section 3]{Heinonen:Kilpelainen:Martio:1993}) implies that
	\begin{align*}
	\|\phi_i\|_{L^\infty(B_{r/4})}\leq Cr^{-1}\|\phi_i\|_{L^2(B_{r/2})}
	\end{align*}
	and that
	\begin{align*}
	\|\phi_i\|_{L^\infty(B_6)}\leq C\|\phi_i\|_{L^2(B_7)}.
	\end{align*}
	Substituting these estimates in~\eqref{eq:three circle}, we have
	\begin{align*}
	\|\phi_i\|_{L^\infty(B_1)}\leq C\big(r^{-1}\|\phi_i\|_{L^2(B_{r/2})}\big)^{\theta}\|\phi_i\|_{L^2(B_7)}^{1-\theta},
	\end{align*}
	from which our claim follows.
\end{proof}

Finally, we need the following Caccioppoli's inequality.
\begin{lemma}\label{cacciWE}
	Let $v$ be a real solution of the p-Laplace equation \eqref{PLapLOGT}, $1<p<\infty$, with the lower order gradient term. Then the following Caccioppoli's inequality holds true.
	\[
	\int_{B_r}\abs{\nabla v}^p dx \leq \frac{C\|A\|_{L^p(B_8)}^{p}}{(\rho-r)^p} \|v\|_{L^{\infty}(B_{\rho})}^{p},
	\]
	for $0<r<\rho<8$, where the positive constant $C$ depends only on $p$.
\end{lemma}

\begin{proof}
	Let $\eta\in C_0^\infty(B_8)$ be a smooth cut-off function, i.e., $0\leq \eta\leq 1$ in $B_\rho$, $\eta\equiv 0$ in $B_8\backslash \overline{B_\rho}$, $\eta\equiv 1$ in $B_r$ and $|\nabla \eta|\leq \frac{4}{\rho-r}$ in $B_\rho\backslash \overline{B_r}$. Then $v\eta^p\in W^{1,p}_0(B_8)$ is an admissible test function and so
	\begin{align*}
	\int_{B_8}\Big(\dive\big(|\nabla v|^{p-2}\nabla v\big)+|\nabla v|^{p-2}A\cdot \nabla v\Big)v\eta^p dx=0.
	\end{align*}
	Integration by parts, we have
	\begin{align*}
	-\int |\nabla v|^{p-2}\nabla v\cdot \big(\nabla v\eta^p+pv\eta^{p-1}\nabla \eta\big)dx+\int |\nabla v|^{p-2}A\cdot \nabla v v\eta^pdx=0.
	\end{align*}
	It follows
	\begin{align*}
	\int |\nabla v|^p\eta^pdx&=-p\int v|\nabla v|^{p-2}\eta^{p-1}\nabla v\cdot \nabla \eta dx+\int |\nabla v|^{p-2}A\cdot \nabla v v\eta^pdx\\
	&=: I_1+I_2.
	\end{align*}
	By H\"older's inequality,
	\begin{align*}
	|I_1|&\leq p\int |\eta \nabla v|^{p-1}|v\nabla \eta|dx\\
	&\leq p\Big(\int \eta^p|\nabla v|^pdx\Big)^{1-\frac{1}{p}}\Big(\int |v|^p|\nabla \eta|^pdx\Big)^{\frac{1}{p}}.
	\end{align*}
	Similarly, 
	\begin{align*}
	|I_2|&\leq \int \eta|A||v| |\eta\nabla v|^{p-1}dx\\
	&\leq \Big(\int \eta^p|\nabla v|^pdx\Big)^{1-\frac{1}{p}}\Big(\int |v|^p|A|^p \eta^pdx\Big)^{\frac{1}{p}}.
	\end{align*}
	Combining these estimates, we obtain
	\begin{align*}
	\int |\nabla v|^{p}\eta^pdx&\leq p^p\int |v|^p|\nabla \eta|^pdx+\int |v|^p|A|^p \eta^pdx\\
	&\leq \frac{c(p)}{(\rho-r)^p}\|v\|_{L^\infty(B_\rho)}^p+\|v\|_{L^\infty(B_\rho)}^p\|A\|_{L^p(B_8)}^p\\
	&\leq \frac{c(p)\|A\|_{L^p(B_8)}^{p}}{(\rho-r)^p} \|v\|_{L^{\infty}(B_{\rho})}^{p},
	\end{align*}
	from which our lemma follows.	
\end{proof}

\subsection{Proofs of the main results}
With all the necessary auxiliary results at hand, the proofs of our following main results go along the same line as~\cite[Section 2]{Kenig:Wang:2015}.

\begin{theorem}\label{CaseiThe}
	Let $v \in W_{loc}^{1,p}(B_8)$ be a real solution of  \eqref{PLapLOGT}. Assume that $v$ satisfies $\|v\|_{L^{\infty}(B_8)}\leq C_0$ and $\sup_{z\in B_1}\abs{\nabla v(z)} \geq 1$. Then
	\begin{equation}\label{esrir}
	\|v\|_{L^{\infty}(B_r)} \geq C_{0}^{C_1} r^{C_2\log C_0} r^{C_3M}
	\end{equation}
	where $C_1, C_2, C_3$ are positive constants depending only on $p$ and $q$.
\end{theorem}

\begin{proof}
	Set
	\begin{equation*}
	\delta=
	\begin{cases}
	q & \text{if } q< 1+\frac{1}{k} \\
	\frac{3k+1}{2k} & \text{otherwise, } 
	\end{cases}
	\end{equation*}	
	since $\max\{2,p\}<q$ and $0\leq k<1$, $\delta\in (2,1+\frac{1}{k})$.
	
	Recall that, $F= \abs{G}^{\frac{p-2}{2}}G$, where $G= v_x - iv_y $ and that $v$ satisfies \eqref{PLapLOGT}. 
	Due to Theorem \ref{ThpLLL}, the nonlinear function $F$ satisfies the following quasilinear Beltrami equation
	\[
	\frac{\doo F}{\doo\overline{z}} = q_1 \frac{\doo F}{\doo z} + q_2 \overline{\frac{\doo F}{\doo z}} + q_3 F,
	\]
	where $q_1, q_2, q_3$ are Beltrami coefficients given as in Theorem \ref{ThpLLL}. Moreover, each solution of this Beltrami equation is of the form
	\[
	F(z) = h(\phi(z))e^{\omega(z)} \  \ \text{in}\ B_8
	\]  
	with $\phi : B_8 \rightarrow \phi(B_8)$ is a $K$-quasiconformal mapping with $K$ depending only on $p$, $h:\phi(B_8)\to \R^2$ is holomorphic, and $\omega$ is the Cauchy transform of certain $L^\delta$-function $g$ that solves \eqref{integral}.
	
	By the Caccioppoli's inequality, Lemma \ref{cacciWE}, we have 
	\[
	\begin{split}
	\int_{B_r}\abs{\nabla v}^p dx 
	&\leq \frac{C\|A\|_{L^p(B_8)}^{p}}{(\rho-r)^p} \|v\|_{L^{\infty}(B_{\rho})}^{p}  \\ 
	&\leq \frac{C\|A\|_{L^q(B_8)}^{p}}{(\rho-r)^p} \|v\|_{L^{\infty}(B_{\rho})}^{p}, \  \ 0<r<\rho<8.
	\end{split}
	\]
	Since $\|A\|_{L^q(B_8)} \leq M$, $M\geq 1$, we have
	\begin{equation}\label{proofCaaci}
	\int_{B_r}\abs{\nabla v}^p dx \leq \frac{CM^p}{(\rho-r)^p} \|v\|_{L^{\infty}(B_{\rho})}^{p}.
	\end{equation}
	
	Note that, the Cauchy transformation $T(g(z)) = \omega(z)$ is a bounded linear operator from $L^\delta(B_8)$ to $L^{\infty}(B_8)$ for $2<\delta<\infty$ (cf.~\cite[Section 4.3.2]{Astala:Iwaniec:Martin}) and that the Beltrami operator $I-\hat{o}_1S$ has quantitative bounded inverse norm from $L^\delta(B_8)$ to $L^\delta(B_8)$ (cf.~\cite[Theorem 1]{Astala:Iwaniec:Saksman:2001}) for all $\delta\in (2,1+\frac{1}{k})$. It follows that $$\|\omega\|_{L^\infty(B_8)}\leq c(q)\|g\|_{L^\delta(B_8)}\leq c(q,k)\|q_3\|_{L^\delta(B_8)}\leq c(q,k)\|A\|_{L^q(B_8)}\leq c(q,k)M.$$
	Now, using the Caccioppoli's inequality above, we have
	\[
	\begin{split}
	\|h\circ\phi\|_{L^2(B_{r/2})}^{2} 
	& = \int_{B_{r/2}}\abs{F}^2\abs{e^{-\omega(z)}}^2 \\
	& \leq C e^{2CM} \int_{B_{r/2}} \abs{\nabla v}^p \\
	& \leq e^{2CM} \frac{C^2M^p}{(r/2)^p} \|v\|_{L^{\infty}(B_r)}^{p},
	\end{split}
	\]
	where $C$ is a positive constant depending only on $p$ and $q$ (since $k$ depends only on $p$). Similarly,
	\[
	\|h\circ\phi\|_{L^2(B_7)}^{2} \leq C^2 e^{2CM} M^p \|v\|_{L^{\infty}(B_8)}^{p}.
	\]
	Since $h$ is holomorphic and $\phi$ is $K$-quasiconformal, the composition $h\circ\phi$ is $K$-quasiregular. Replacing $\phi$ by $h\circ\phi$ in Proposition \ref{prop:key prop} yields
	\begin{align*}
	\|h\circ\phi\|_{L^\infty(B_1)}^{2}\leq C\big(r^{-2}\|h\circ\phi\|_{L^2(B_{r/2})}^2\big)^{\theta}\|h\circ\phi\|_{L^2(B_7)}^{2(1-\theta)},
	\end{align*} 
	with $0<\theta<1$ satisfying $\theta=\frac{E}{\log\frac{E'}{r}}$ and $r\in (0, \frac{1}{4})$, where $E, E'$ are the positive constants depending only on $K$ (and thus only on $p$). As a consequence,
	\begin{equation}\label{EsTTima}
	e^{-2CM}\|\nabla v\|_{L^{\infty}(B_1)}^{p}
	\leq C \left[r^{-2}e^{2CM}\frac{C^2M^p}{r^p}\|v\|_{L^{\infty}(B_r)}^{p}\right]^{\theta} \left[C^2e^{2CM}M^p\|v\|_{L^{\infty}(B_8)}^{p} \right]^{1-\theta}.  
	\end{equation}
	Recall that we have the following assumptions on $v$:
	\begin{itemize}
		\item[(1)]
		$\abs{v(z)} \leq C_0$ for all $z\in B_8$. \\
		\item[(2)] $\sup_{z\in B_1} \abs{\nabla v(z)} \geq 1$. \\
		\item[(3)] $v\in W_{loc}^{1,p}(B_8)$.
	\end{itemize}
	Therefore, the estimate \eqref{EsTTima} becomes,
	\[
	\|v\|_{L^{\infty}(B_r)} \geq \frac{r^{1+\frac{2}{p}}e^{-\frac{4CM}{\theta p}}}{[MC^{\frac{1}{p}}C_0^{1-\theta}]^{\frac{1}{\theta}}}.
	\]
	Substituting the explicit expression of $\theta = \frac{E}{\log(\frac{E'}{r})}$ in the above estimate, we obtain that
	\[
	\|v\|_{L^{\infty}(B_r)} \geq C_{0}^{C_1} r^{C_2\log C_0} r^{C_3M},
	\]
	where $C_1, C_2, C_3$ are positive constants depending only on $p$ and $q$.	
\end{proof}

With the aid of Theorem \ref{CaseiThe}, we are ready to prove our main result of this section. The idea of the proof is similar to that used in~\cite[Corollary 3.3]{Kenig:Wang:2015}.

\begin{theorem}\label{mainpq2}
	Let $u\in W_{loc}^{1,p}(\mathbb{R}^2)$, $1<p<\infty$, be a real solution of $$\dive(\abs{\nabla u}^{p-2}\nabla u) + W\cdot(\abs{\nabla u}^{p-2}\nabla u) = 0 \ \text{in} \ \mathbb{R}^2.$$ Assume also that $\|W\|_{L^q(\mathbb{R}^2)}\leq \tilde{M}$ with  $q > \max \{2,p\}$ and $\abs{\nabla u(0)} = 1$.
\begin{itemize}
\item[(i)] If $u$ satisfies $\abs{u(z)}\leq C_0$, then
	\begin{equation}\label{constbound}
	\inf_{\abs{z_0}=R}\sup_{\abs{z-z_0}<1} \abs{u(z)} \geq e^{-CR^{1-\frac{2}{q}}\log R}
	\end{equation}
	for $R\gg 1$, where $C$ depends only on $p$, $q$, $\tilde{M}$ and $C_0$.
\item[(ii)] If $u$ satisfies $\abs{u(z)}\leq\abs{z}^m$ for $\abs{z}>1$ with some $m>0$, then the similar estimate \eqref{constbound} also holds true.
\end{itemize}
\end{theorem}

\begin{proof}
(i). Similar as in~\cite[Corollary 2.2]{Kenig:Wang:2015}, we use the scaling argument from \cite{Bourgain:Kenig:2005}. Let $\abs{z_0} = R$, $R\gg 1$ and we define $u_R(z) = u(Rz + z_0)$. Then $u_R$ satisfies the following weighted $p$-Laplace equation with lower order gradient term:
	\[
	\dive(\abs{\nabla u_R}^{p-2}\nabla u_R) +  W_R\cdot(\abs{\nabla u_R}^{p-2} \nabla u_R) = 0 \ \text{in}\ B_8,
	\]
	where $W_R(z) = RW(Rz+z_0)$. Since $q>\max\{2,p\}$, we have
	\begin{align}
	\left(\int_{B_8}\abs{W_R}^q\right)^{\frac{1}{q}} 
	& \leq R^{1-\frac{2}{q}} \left(\int_{\mathbb{R}^2} \abs{W}^q\right)^{\frac{1}{q}} \nonumber \\
	& \leq \tilde{M} R^{1-\frac{2}{q}}.
	\end{align}
	Note that
	\[
	\abs{\nabla u_R(-\frac{z_0}{R})} = R\abs{\nabla u(0)} = R >1.
	\]
	The desired estimate follows by applying \eqref{esrir} with $M=\tilde{M}R^{1-\frac{2}{q}}$ and $r=R^{-1}$.

	(ii). The proof is similar to part (i). The only difference is that, in this case, we use the polynomial boundedness assumption on the solution $u_R$. Our assumption on $u_R$ implies that,
	\[
	\|u_R\|_{L^{\infty}(B_8)} \leq 9^m R^m
	\]
	i.e., $C_0 = 9^mR^m$. Therefore the estimate \eqref{constbound} can be achieved by replacing $C_0 = 9^mR^m, M=\tilde{M}R^{1-\frac{2}{q}}$ and $r=R^{-1}$ in \eqref{esrir}. 
\end{proof}

As an immediate consequence of Theorem \ref{CaseiThe}, we are now able to prove the SUCP for~\eqref{eq:main equation}.

\begin{proof}[Proof of Theorem~\ref{thm:SUCP}]
	Using translation and scaling if necessary, we may assume that $z_0=0$ and $B_8\subset\Omega$. Note that $\|A\|_{L^\infty(B_8)}=M<\infty$. We prove the claim by a contradiction argument. If $v(z)\not\equiv v(0)$ in $B_1$, then $\sup_{z\in B_1}|\nabla u(z)|\geq C$ for some $C>0.$
	Therefore, applying Theorem \ref{CaseiThe} with $q=\infty$, we infer that $v(z)-v(0)$ can not vanish at $0$ in an infinite order. Hence $v(z)\equiv v(0)$ in $B_1$. Using the standard chain of balls argument, we thus conclude that $v(z)\equiv v(0)$ in $\Omega$.
\end{proof}

\section{SUCP II: $q=\max\{p,2\}$}\label{sec:q=p 2}
In this section, we deal with the case $q=\max\{p,2\}$. It should be noticed that if $2<p<\infty$, then $q=p$ and we may apply the exact argument as in the previous section to show the quantitative uniqueness estimates. Thus, throughout this section, we will assume that $1<p\leq2$. In particular, this implies that $q=2$. As before, we assume that $\|A\|_{L^2(B_8)}\leq M$ for some constant $M>1$.

The main difficulty we face in this case is that the Cauchy transform is not a bounded operator from $L^2(B_8)$ to $L^{\infty}(B_8)$. Instead, we will use the fact that it is a bounded operator from $L^2(B_8)$ to $W^{1,2}(B_8)$. This fact, together with certain Trudinger type inequality from~\cite{Kenig:Wang:2015}, will enable us to repeat the previous arguments to conclude the proof.

\subsection{Auxiliary result}

We will need the following upper Harnack type estimate for quasiregular mappings.

\begin{lemma}\label{lemma:upper Harnack}
	If $\phi:B_8\to \R^2$ is a $K$-quasiregular mapping, then for each $r\in (0,1)$ we have
	\[
	\|\phi\|_{L^{\infty}(B_{r/2})} \leq C\frac{1}{\abs{B_r}}\int_{B_r}\abs{\phi}dx,
	\]
	where the constant $C>0$ depends only on $K$.
\end{lemma}

\begin{proof}
	As noticed in the proof of Proposition~\ref{prop:key prop}, the component functions $\phi_i=x_i\circ \phi$, $i=1,2$, are $\mathcal{A}$-harmonic of type 2. Thus, the standard regularity theory of elliptic PDEs (cf.~\cite[Section 3]{Heinonen:Kilpelainen:Martio:1993}) implies that for each $i=1,2$, we have
	\begin{align*}
	\|\phi_i\|_{L^{\infty}(B_{r/2})} \leq C\frac{1}{\abs{B_r}}\int_{B_r}\abs{\phi_i}dx.
	\end{align*}
	The claim follows by summing up the above estimates.
\end{proof}

\subsection{Proofs of the main results}

\begin{theorem}\label{QUCP_q=2}
	Let $v\in W_{loc}^{1,p}(B_8), 1<p\leq 2$, be a real solution of \eqref{PLapLOGT}.
	If $\abs{v(z)} \leq C_0, (C_0\geq 1)$ for all $z\in B_8$, and $\|\nabla v\|_{L^p(B_{6/5})}\geq 1$, then
	\begin{equation}\label{bounDq2}
	\|v\|_{L^{\infty}(B_r)} \geq C_1 C_0^{-C_2} r^{C_3\log C_0} r^{C_4M^2}
	\end{equation}
	for all  $0<r\ll 1$, where the positive constants $C_1, C_2, C_3, C_4$ depend only on $p$.
\end{theorem}

\begin{proof}
	Recall that, $F$ satisfies the following Beltrami equation
	\begin{equation}\label{Beltrami_q=2}
	\frac{\doo F}{\doo\bar{z}} = q_1 \frac{\doo F}{\doo z} + q_2 \overline{\frac{\doo F}{\doo z}} + q_3 F \ \text{in}\ B_8,
	\end{equation}
	where $k=\|q_1\|_{L^\infty(B_8)}+\|q_2\|_{L^\infty(B_8)}<1$ with $k=\frac{2-p}{3p-2}$ and $\|q_3\|_{L^2(B_8)} \leq C M$, where $C$ depends on $p$. The solution of \eqref{Beltrami_q=2} is represented by
	\[
	F(z) = (h\circ\phi)(z) e^{-\omega(z)} \  \  \text{in}\ B_8,
	\]
	where $h$ is holomorphic in $\phi(B_8)$, $\phi$ is $K$-quasiconformal and 
	\[
	-\omega(z) = (Tg)(z) = -\frac{1}{\pi} \int_{B_8}\frac{g(\xi)}{\xi-z} d\xi, \  \ g\in L^2(B_8),
	\]
	see for instance the proof of Theorem \ref{ThpLLL}. Due to the fact that Cauchy transform is bounded from $L^2(B_8)$ to $W^{1,2}(B_8)$ (cf.~\cite[Section 4.3.2]{Astala:Iwaniec:Martin}) and that the operator $I-\hat{o}_1S$ is invertible from $L^2(B_8)$ into $L^2(B_8)$, we obtain
	\[
	\|\omega\|_{W^{1,2}(B_8)} \leq C\|g\|_{L^2(B_8)} \leq C\|q_3\|_{L^2(B_8)} \leq CM.
	\]
	In the second inequality we used the property that, the following integral equation, see \eqref{integral},
	\[
	\left( I-\hat{o}_1S\right)g = \hat{o}_2, \ \text{in}\ \mathbb{R}^2,
	\]
	has a unique solution in $L^2(\mathbb{R}^2)$ and the solution is bounded by $q_3$ in terms of the $L^2(\mathbb{R}^2)$ norm. With the above representation of the solution and the auxiliary lemmata at hand, we proceed our proof similar to \cite[Section 3]{Kenig:Wang:2015}.
	
	 Applying the Hadamard's three circle theorem (cf.~the proof of Proposition \ref{prop:key prop}) to the $K$-quasiregular mapping $(h\circ\phi)(z) = F(z)e^{\omega(z)}$, we have
	\begin{equation}\label{hadamarda}
	\|Fe^{\omega}\|_{L^{\infty}(B_{6/5})} \leq \|Fe^{\omega}\|_{L^{\infty}(B_{r/4})}^{\theta} \|Fe^{\omega}\|_{L^{\infty}(B_2)}^{1-\theta},
	\end{equation}
	where
	\begin{equation}\label{formtheta}
	0<\theta=\theta(r) = \frac{\log\frac{2}{2-(\frac{4}{5C_0})^{1/\alpha}}}{\log\frac{2}{(\frac{r}{4C_0})^{1/\alpha}}}<1, \quad \alpha\in (0, 1],
	\end{equation}
	with $C_0$ being a positive constant depending on $K$ (which is in particular depending on $p$). Using Lemma \ref{lemma:upper Harnack} and H\"older's inequality, we obtain that
	\begin{align*}
	\|Fe^{\omega}\|_{L^{\infty}(B_{r/4})} 
	& =\|h\circ\phi\|_{L^{\infty}(B_{r/4})} \\
	& \leq \frac{C}{\abs{B_{r/2}}} \int_{B_{r/2}}\abs{F(z)e^{\omega(z)}} dz \\
	& \leq C \left( \frac{1}{\abs{B_{r/2}}} \int_{B_{r/2}} e^{2\abs{\omega}}\right)^{1/2}
	\left( \frac{1}{\abs{B_{r/2}}} \int_{B_{r/2}} \abs{F}^2\right)^{1/2}.
	\end{align*}
	Note that, by~\cite[Lemma 3.3]{Kenig:Wang:2015}, we have the following integral estimate:
	\[
	\frac{1}{\abs{B_{r/2}}} \int_{B_{r/2}} e^{2\abs{\omega}} \leq C r^{-2CM} e^{2CM^2}.
	\]
	 Therefore, applying Caccioppoli's inequality, Lemma \ref{cacciWE}, we conclude that
	\begin{align*}
	\|Fe^{\omega}\|_{L^{\infty}(B_{r/4})} 
	& \leq C r^{-(C_1+CM)} e^{{CM}^2} \left(\int_{B_{r/2}} \abs{\nabla v}^p\right)^{1/2}, \ \ 1<p\leq 2, \\
	& \leq C r^{-(C_1+CM)} e^{{CM}^2} \frac{\|A\|_{L^p(B_8)}^{p/2}}{(r/2)^{p/2}} \|v\|_{L^{\infty}(B_r)}^{p/2} \\
	&\leq C M r^{-(C_1+CM)} e^{{CM}^2}\|v\|_{L^{\infty}(B_r)}^{p/2},
	\end{align*}
	where the constant $C_1>0$ depends only on $p$. 
	Substituting the above estimate in \eqref{hadamarda}, the right hand side of the Hadamard's three circle theorem becomes
	\begin{align}
	\|Fe^{\omega}\|_{L^{\infty}(B_{r/4})}^{\theta} \|Fe^{\omega}\|_{L^{\infty}(B_{2})}^{1-\theta} 
	\leq &\left( C M r^{-(C_1+CM)} e^{{CM}^2}\|v\|_{L^{\infty}(B_r)}^{p/2}\right)^{\theta} \nonumber \\
	&\times \left( C M 8^{-(C_1+CM)} e^{{CM}^2}\|v\|_{L^{\infty}(B_8)}^{p/2}\right)^{1-\theta}. \label{hadaMCaccio}
	\end{align}
	On the other hand, we have
	\begin{align}
	1 \leq \|\nabla v\|_{L^p(B_{6/5})}^p \nonumber
	& = \int_{B_{6/5}} \abs{F}^2 \nonumber \\
	& \leq \int_{B_{6/5}} \abs{h\circ\phi}^2 \abs{e^{2\abs{\omega}}} \nonumber\\
	& \leq \left( \int_{B_{6/5}} e^{4\abs{\omega}}\right)^{1/2} \|h\circ\phi\|_{L^4(B_{6/5})}^{2} \nonumber \\ 
	& \leq Ce^{2CM^2} \|h\circ\phi\|_{L^4(B_{6/5})}^{2} \nonumber \\
	& \leq Ce^{2CM^2} \|h\circ\phi\|_{L^{\infty}(B_{6/5})}^{2}. \label{esiMaTe}
	\end{align}
	Finally, our claim follows by combining the estimates \eqref{hadamarda}, \eqref{hadaMCaccio}, \eqref{esiMaTe} and the form of $\theta$ (cf.~\eqref{formtheta}).
\end{proof}

\begin{theorem}\label{thm:q=2}
	Let $u\in W_{loc}^{1,p}(\mathbb{R}^2), \ 1<p\leq 2,$ be a real solution of the equation
	\[
	\dive(\abs{\nabla u}^{p-2}\nabla u) + W \cdot(\abs{\nabla u}^{p-2}\nabla u) = 0 \ \text{in}\  \mathbb{R}^2.
	\]
	Assume that $\|W\|_{L^2(\mathbb{R}^2)} \leq \ \tilde{M}$ and $\|\nabla u\|_{L^p(B_1)}\geq 1$.
	\begin{itemize}
	\item[(i)] If $u$ satisfies $\abs{u(z)}\leq C_0$ for some $C_0>0$, then
	\[
	\inf_{\abs{z_0}=R} \sup_{\abs{z-z_0}<1} \abs{u(z)} \geq R^{-C},
	\]
	for $R\gg1$, where $C>0$ depends only on $\tilde{M}, p$ and $C_0$.
\item[(ii)] If $u$ satisfies $\abs{u(z)}\leq \abs{z}^m$ for $\abs{z}>1$ with some $m>0$. Then
\[
\inf_{\abs{z_0}=R} \sup_{\abs{z-z_0}<1} \abs{u(z)} \geq C_1 e^{-C_2(\log R)^2},
\]
for $R\gg1$, where $C_1>0$ depends only on $m, p$ and $C_2>0$ depends only on $m, p, \tilde{M}$. 
\end{itemize}
\end{theorem}

\begin{proof}
(i). Let $\abs{z_0}=R, R\gg1.$ Define $v(z) = u(Rz+z_0)$. Then $v$ solves
	\[
	\dive(\abs{\nabla v}^{p-2}\nabla v) + W_R \cdot(\abs{\nabla v}^{p-2} \nabla v) = 0  \  \ \text{in} \ B_8
	\]
	with $W_R(z) = RW(Rz+z_0).$ Since $\|W\|_{L^2(\R^2)} \leq \tilde{M}$, we also have 
	$\|W_R\|_{L^2(B_8)} \leq \tilde{M}$. The final proof is similar to~\cite[Theorem 1.1, Part (ii)]{Kenig:Wang:2015}.
	
	(ii). Under the assumption on $u_R$, we have,
	\[
	\|u_R\|_{L^{\infty}(B_8)} \leq 9^mR^m,
	\]
	i.e., $C_0=9^mR^m$. Finally, replacing $r=\frac{1}{R}, C_0=9^mR^m$ in \eqref{bounDq2}, we obtain
	
	\begin{align*}
	\|u\|_{L^{\infty}(B_{1}(z_0))} 
	& = \|u_R\|_{L^{\infty}(B_r(0))} \\
	& \geq C_1(9^mR^m)^{-C_2}R^{-C_3\log (9^mR^m)} R^{-C_4\tilde{M}^2} \\
	& \geq \tilde{C}_1 e^{(-\tilde{C}_2(\log R)^2)}, \ \text{for}\ R\gg 1
	\end{align*}
	where $\tilde{C}_1$ depends on $m,p$ and $\tilde{C}_2$ depends on $m,p, \tilde{M}$.
\end{proof}

\section{SUCP for the weighted $p$-Laplace equation}\label{sec:weighted p laplace}

In this section, we consider the weighted $p$-Laplace equation, $p\in(1,\infty)$, 
\begin{equation}\label{weighted_pLalLipsc}
\dive(\sigma \abs{\nabla u}^{p-2} \nabla u) = 0 \ \text{ in } \ \Omega,
\end{equation}
where $\sigma\in W_{loc}^{1,\infty}(\Omega)$ is a locally positive locally Lipschitz continuous function, i.e., $\sigma$ is locally Lipschitz and for each $K\subset\subset \Omega$,  there exists a positive constant $c_K$ such that $\sigma>c_K$ in $K$. 

The main idea of the proof relies on localising the problem on a small region and then one uses the scaling argument as discussed in the previous sections.
We show that the estimate of the maximal vanishing order of the solution $v$ to the equation
\begin{equation}\label{eq1}
\dive(A \abs{\nabla v}^{p-2} \nabla v) = 0 \ \text{ in } \ B_8,
\end{equation}
with $A$ being a positive Lipschitz coefficient, would be enough to establish Theorem \ref{thm:SUCP weighted p-la}. Similar to the previous case, we first reduce the problem into certain quasilinear Beltrami equation. Then the derivation of the estimate of the maximal vanishing order for $v$ follows from the explicit solution of the Beltrami type equation, the appropriate use of Caccioppoli's inequality for the $p$-Laplace equation and the Hadamard's three circles theorem.

\subsection{Auxiliary result}
The following Caccioppoli's inequality will be useful in our later proofs.
\begin{lemma}\label{cacciWE2}
	Let $v$ be a weak solution of \eqref{eq1}. Then the following Caccioppoli's inequality holds true:
	\[
	\int_{B_r}\abs{\nabla v}^p dx \leq \frac{C\|A\|_{L^p(B_8)}^{p}}{(\rho-r)^p} \|v\|_{L^{\infty}(B_{\rho})}^{p},
	\]
	for $0<r<\rho<8$, where $C$ depends only on $p$.
\end{lemma}

\begin{proof}
	Since $v$ is a weak solution of the weighted $p$-Laplace equation \eqref{eq1},
	\begin{equation}\label{cacci}
	\int_{B_8} A\abs{\nabla v}^{p-2}\nabla v\cdot\nabla\eta dx =0
	\end{equation}
	for all $\eta\in W_{0}^{1,p}(B_8)$. 
	
	Let $\zeta\in C_0^\infty(B_8)$ be a smooth cut-off function, i.e., $0\leq \zeta\leq 1$ in $B_\rho$, $\zeta\equiv 0$ in $B_8\backslash \overline{B_\rho}$, $\zeta\equiv 1$ in $B_r$ and $|\nabla \zeta|\leq \frac{4}{\rho-r}$ in $B_\rho\backslash \overline{B_r}$. Set
	\begin{align*}
	\eta &= \zeta^p v, \\
	\nabla\eta & = \zeta^p\nabla v + p\zeta^{p-1} v \nabla \zeta.
	\end{align*}
	Then H\"older's inequality implies
	\[
	\begin{split}
	\int_{B_8} \zeta^p \abs{\nabla v}^p dx 
	& = -p \int_{B_8} A\zeta^{p-1}v \abs{\nabla v}^{p-2}\nabla v \cdot \nabla\zeta dx \\
	& \leq p \int_{B_8} \abs{\zeta\nabla v}^{p-1} \abs{Av\nabla\zeta} dx \\
	& \leq p \left\{\int_{B_8}\zeta^p \abs{\nabla v}^p dx\right\}^{1-1/p} 
	\left\{\int_{B_8}\abs{A}^p\abs{v}^p \abs{\nabla\zeta}^p dx\right\}^{1/p},
	\end{split}
	\]
	i.e.,
	\begin{equation}\label{cacci1}
	\int_{B_8} \zeta^p \abs{\nabla v}^p dx \leq p^p \int_{B_8}\abs{A}^p\abs{v}^p \abs{\nabla\zeta}^p dx.
	\end{equation}
	In particular,
	\[
	\begin{split}
	\int_{B_r} \abs{\nabla v}^p dx 
	& \leq \frac{p^p}{(\rho-r)^p}\int_{B_{\rho}}\abs{A}^p\abs{v}^p dx \\
	& \leq \frac{p^p}{(\rho-r)^p} \|v\|_{L^{\infty}(B_\rho)}^{p} \int_{B_\rho} \abs{A}^p dx \\
	& \leq \frac{C \|A\|_{L^p(B_8)}^{p}}{(\rho-r)^p} \|v\|_{L^{\infty}(B_\rho)}^{p},
	\end{split}
	\]
	where $0<r<\rho<8$ and the constant $C>0$ depends only on $p$.
\end{proof}

We follow the computation from~\cite[Appendix A3]{Guo:Kar:Salo:2015} to transform the weighted $p$-Laplace equation \eqref{eq1} to a certain Beltrami type equation. Let $G = A v_x - iA v_y$, where $v$ solves \eqref{eq1} and define $F=\abs{G}^aG$ with $a=\frac{p-2}{2}$. Then $F$ satisfies
\begin{align}\label{eq:for Fa}
\frac{\partial F}{\partial \bar{z}}=q_1\frac{\partial F}{\partial z}+q_2\overline{\frac{\partial F}{\partial z}}+q_3 F, \  \text{in}\ B_8,
\end{align}
where
\[
q_1 = -\frac{1}{2}\left(\frac{p-2}{p+2} + \frac{p-2}{3p-2} \right)\frac{\overline{F}}{F},
\]
\[
q_2 = -\frac{1}{2}\left(\frac{p-2}{3p-2} - \frac{p-2}{p+2} \right)\frac{F}{\overline{F}}
\]
and
\[
\begin{split}
q_3
=&  A \frac{p}{p+2}\left[\frac{\overline{F}}{F} \frac{\partial}{\partial z}\left(\frac{1}{A}\right) -  \frac{\partial}{\partial \bar{z}}\left(\frac{1}{A}\right)\right] \\
&- A^{p-2} \frac{p}{3p-2}\left[ \frac{\overline{F}}{F} \frac{\partial}{\partial z}\left(\frac{1}{A^{p-2}}\right) + \frac{\partial}{\partial \bar{z}}\left(\frac{1}{A^{p-2}}\right)\right].
\end{split}
\]
Since $A\geq c_0$ and $A$ is Lipschitz in $B_8$, it is easy to check that $k=\|q_1\|_{L^{\infty}(B_8)}+\|q_2\|_{L^{\infty}(B_8)}<1$ and $\|q_3\|_{L^{\infty}(B_8)}\leq M$ (we may assume that $M>1$). For $\delta\in (2,1+\frac{1}{k})$, we know that $\|q_3\|_{L^{\delta}(B_8)}\leq CM$, where $k$ is explicitly calculated in Remark \ref{rmk:on the k} and the constant $C>0$ depends eventually only on $p$. Therefore, by the proof of Theorem \ref{ThpLLL}, the solution of the Beltrami equation \eqref{eq:for Fa} has the following representation:
\[
F(z) = (h\circ\phi)(z)e^{\omega(z)} \ \ \text{in}\ B_8
\]
where $\phi:B_8\to \phi(B_8)$ is $K$-quasiconformal (with $K$ depending only on $p$), $h:\phi(B_8)\to \R^2$ is holomorphic, and $\omega(z)=(Tg)(z)$ is the Cauchy transform of $g$ for some $g\in L^{\delta}(B_8)$ with $2<\delta<(1+\frac{1}{k})$.

\subsection{Proofs of the main results}\label{subsec:Assuming Lipschitz regularity}

We are now able to prove the following estimate of the maximal vanishing order for the solution $v$ of~\eqref{eq1}.

\begin{theorem}\label{thm:maximal vanishing order}
	Let $v\in W_{loc}^{1,p}(B_8)$ be a weak solution of the weighted $p$-Laplace equation \eqref{eq1} with $A$ satisfying that $A\geq c_0$ and that $A$ is $M$-Lipschitz in $B_8$. If $\|v\|_{L^{\infty}(B_8)}\leq C_0$ and $\sup_{z\in B_1}\abs{\nabla v(z)}\geq 1$, then
	\begin{equation}\label{lower_bound_infty}
	\|v\|_{L^{\infty}(B_r)} \geq C_{0}^{C_1} r^{C_2\log C_0} r^{C_3M}
	\end{equation}
	where $C_1, C_2, C_3$ are positive constants depending only on $p$, $c_0$, and $M$.
\end{theorem}

\begin{proof}
	The proof is similar to Theorem \ref{CaseiThe}. For the convenience of the readers, we would like to add a few lines here. In the view of the mapping properties of the Cauchy transform and the operator
	$I-\hat{o}_1S$, we see that
	\begin{equation}\label{pLaplace_LipscTH}   
	\abs{\omega(z)} \leq CM, \ \text{for}\ z\in B_8.
	\end{equation}
	Also, using the Caccioppoli's inequality, see Lemma \ref{cacciWE2}, we have the following estimate for the $K$-quasiregular map $h\circ\phi$,
	\begin{equation}\label{Quasi_Cacci}
	\|h\circ\phi\|_{L^2(B_{r/2})}^{2} \leq e^{2CM} \frac{C^2M^p}{(r/2)^p} \|v\|_{L^{\infty}(B_r)}^{p},
	\end{equation}
	where $C>0$ depends only on $p$. We also remark from Proposition \ref{prop:key prop} that
	\begin{align}\label{Hada_pLapl}
	\|h\circ\phi\|_{L^\infty(B_1)}^{2}\leq C\big(r^{-2}\|h\circ\phi\|_{L^2(B_{r/2})}^2\big)^{\theta}\|h\circ\phi\|_{L^2(B_7)}^{2(1-\theta)},
	\end{align} 
	with an appropriate $\theta$ between $0$ and $1$.
	
	Combining \eqref{pLaplace_LipscTH}, \eqref{Quasi_Cacci} and \eqref{Hada_pLapl}, we can deduce the estimate \eqref{EsTTima}. Finally, our result follows based on the estimate \eqref{EsTTima}.
\end{proof}

\begin{proof}[Proof of Theorem~\ref{thm:SUCP weighted p-la}]
	With the help of Theorem~\ref{thm:maximal vanishing order}, the proof is the same as the proof of Theorem~\ref{thm:SUCP}.
\end{proof}

\begin{remark}\label{rmk:on QUCP for p-Laplace type equations}
	With Theorem~\ref{thm:maximal vanishing order} at hand, we could also obtain a version of Theorem~\ref{thm:main thm} for the weighted $p$-Laplace equation~\eqref{eq:weighted p laplace} under certain global a priori boundedness assumption on the weak solution. In particular, for the standard $p$-Laplace equation, this will give certain lower bound on the decay rate of $p$-harmonic functions, which might be of independent interest.
\end{remark}

\subsection{Weakening the regularity on $\sigma$ when $1<p\leq 2$}\label{subsec:weaken the regularity}

In this section, we only consider the case $p\in (1,2]$. The aim is to prove strong unique continuation principle (SUCP) for the weighted $p$-Laplace equation \eqref{eq1} under weaker assumptions on the weight $\sigma$. To be more precise, throughout this section, we assume that  
\begin{enumerate}
	\item $\sigma$ is locally positive, i.e., for each compact set $K$ in $\Omega$, there exists a positive constant $c_K>0$ such that $\sigma>c_K$ in $K$;
	
	\item $\sigma\in L^\infty_{loc}(\Omega)$ if $p=2$ and $\sigma\in C^{0,\alpha}_{loc}(\Omega), 0<\alpha<1$ if $p\in (1,2)$.
\end{enumerate}

\begin{theorem}\label{thm:SUCP 1<p<2}
	 Let $v\in W_{loc}^{1,p}(\Omega)$ be a weak solution of \eqref{weighted_pLalLipsc} with $1<p\leq 2$. If for some $z_0\in \Omega$ and for all $N\in \mathbb{N}$, there exist $C_N>0$ and $r_N>0$ such that
	 \[
	 \abs{v(z)-v(z_0)} \leq C_N\abs{z-z_0}^N, \ \forall \ \abs{z-z_0}<r_N,
	 \]
	 then
	 \[
	 v(z) \equiv v(z_0).
	 \]
\end{theorem}

As in the previous section, the proof is based on the localization argument and we will do much of the analysis to the equation~\eqref{eq1}, where $A$ satisfies 
\begin{enumerate}
	\item $A>c_0$ in $B_8$;
	
	\item $A\in L^\infty(B_8)$ if $p=2$ and $A\in C^{0,\alpha}(\overline{B_8})$ if $p\in (1,2)$.
\end{enumerate}

We start with the following lemma concerning the existence of weighted $q$-harmonic conjugate.

\begin{lemma}\label{conjugate}
	Let $v\in W^{1,p}(B_8)$ be a nontrivial weak solution of~\eqref{eq1}. Then there exists a weighted $q$-harmonic function $w\in W^{1,q}(B_8)$ satisfying $\dive({A}^{1-q}|\nabla w|^{q-2}\nabla w) =0$ such that
	\[
	w_x = -A|\nabla v|^{p-2}v_y \ \text{and}\ w_y = A|\nabla v|^{p-2}v_x,
	\] 
	where $p$ and $q$ are the conjugate exponents.
\end{lemma}

\begin{proof}
	For non constant $v$, the equation $\text{div}(A|\nabla v|^{p-2}\nabla v) =0$ can be written as
	\begin{equation}\label{complex}
	\frac{\doo}{\doo x}\left( A|\nabla v|^{p-2}v_x\right) + \frac{\doo}{\doo y}\left(A|\nabla v|^{p-2}v_y\right)= 0.
	\end{equation}
	Introduce $\phi$ such that $\phi = (\phi_1, \phi_2)$ where
	\[
	\phi_1 = - A|\nabla v|^{p-2}v_y \ \text{and}\ \phi_2=A|\nabla v|^{p-2}v_x.
	\]
	Then \eqref{complex} becomes
	\[
	\frac{\doo \phi_1}{\doo y} = \frac{\doo \phi_2}{\doo x}.
	\]
	Hence, from Cauchy-Riemann equation, there exists $w\in W_{loc}^{1,q}(\Omega)$ with $w(0)=0$ (unique upto constant) such that
	\[
	\frac{\doo w}{\doo x} = \phi_1 \ \text{and}\ \frac{\doo w}{\doo y} = \phi_2,
	\]
	from which the result follows.
\end{proof}

Note that if $p=2$, then both $v$ and $w$ lie in $W^{1,2}(B_8)$. If $p\in (1,2)$, then $v\in C^{1,\beta}(\overline{B_8})$ (cf.~\cite{Lieberman:1988}) and hence both $v$ and $w$ belong to $W^{1,2}(B_8)$. In particular, the mapping $F:=v+iw$ belongs to $W^{1,2}(B_8)$.

\begin{lemma}\label{nonBel}
	Let $v$ and $w$ be given as in Lemma \ref{conjugate}. Then we have the following two identities:
	\begin{enumerate}
		\item [(i)]
		\[
		\frac{\doo F}{\doo \ol{z}} = \frac{1- A|\nabla v|^{p-2}}{1 + A|\nabla v|^{p-2}} \ol{\frac{\doo F}{\doo z} },
		\]
		where $F$ is defined as $F := v + iw.$
		\item [(ii)]
		\[
		\frac{\doo F}{\doo \ol{z}} = \frac{1- A |f|^{p-2}}{1+A|f|^{p-2}}  \frac{\ol{f}}{f} \frac{\doo F}{\doo z},
		\]
		where $f$ is defined as $f :=v_x - iv_y \neq 0.$
	\end{enumerate}
\end{lemma}

\begin{proof}
	Define $F := v + iw$. Calculating the complex derivatives of $F$, we obtain
	\[
	\begin{split}
	\frac{\doo F}{\doo \ol{z}}
	& = \frac{\doo v}{\doo \ol{z}} + i\frac{\doo w}{\doo \ol{z}} \\
	& = \frac{1}{2} (1-A\abs{\nabla v}^{p-2})(v_x+i v_y)
	\end{split}
	\]
	and 
	\[
	\begin{split}
	\frac{\doo F}{\doo z}
	& = \frac{\doo v}{\doo {z}} + i\frac{\doo w}{\doo {z}} \\
	& = \frac{1}{2} (1 + A\abs{\nabla v}^{p-2})(v_x - i v_y).
	\end{split}
	\]
	Combining the above two identities the proof of (i) follows immediately.
		
	In order to establish the second equation, simply notice that	$v_x-iv_y = -i A\abs{f}^{p-2}f$ and that $v_x+iv_y = i A\abs{f}^{p-2}f$. The claim follows by substituting these two identities in (i).
\end{proof}

Note that, $f = v_x - iv_y = 2v_z$ and $\abs{f}=2\abs{v_z}=(v_x^2+v_y^2)^{1/2} = \abs{\nabla v}.$ If we rewrite $F$ as $F= v+iw=Re(F)+iIm(F)$ and $\abs{f} = 2\abs{(Re(F))_z} = \abs{\nabla v}$, then Lemma \ref{nonBel} implies that
\[
\begin{split}
\frac{\doo F}{\doo \ol{z}}
& = \frac{1- 2^{p-2}A\abs{(Re(F))_z}^{p-2}}{1 + 2^{p-2}A\abs{(Re(F))_z}^{p-2}} \ol{\frac{\doo F}{\doo z} }, \\
& = \mu\ol{\frac{\doo F}{\doo z} },
\end{split}
\]
where $\mu$ is a function depending on the complex derivative of the real part of $F$.

\begin{theorem}\label{SolBel1}
	Let $F$ satisfies the following nonlinear Beltrami equation
	\[
	\frac{\doo F}{\doo \ol{z}} = \mu\ol{\frac{\doo F}{\doo z}} \ \text{in}\  B_8,
	\]
	where $\mu$ is defined as
	\[
	\mu := \frac{1- 2^{p-2}A\abs{(Re(F))_z}^{p-2}}{1 + 2^{p-2}A\abs{(Re(F))_z}^{p-2}}.
	\]
	Then each weak solution $F\in W^{1,2}(B_8)$ is represented by
	\[
	F(z) = (h\circ\phi)(z) \ \text{in}\ B_8,
	\]
	where $\phi : B_8 \rightarrow \phi(B_8)$ is a $K$-quasiconformal mapping and $h:\phi(B_8)\to \R^2$ is holomorphic. Moreover, the quasiconformality coefficient $K$ depends only on $p$, $\alpha$, $\|v\|_{L^\infty(B_8)}$, $\frac{\|A\|_{L^\infty(B_8)}}{c_0}$ and $c_0$ when $p\in (1,2)$ and on $c_0$ when $p=2$.
\end{theorem}

\begin{proof}
	\begin{case}
		\it{$p=2$:}
	\end{case}
	In this case, 
	\[
	\mu := \frac{1-A}{1+A}. 
	\]
	In view of our assumptions on $A$, we have $\|\mu\|_{L^\infty(B_8)}\leq \mu_0<1$, where $\mu_0>0$ is a constant depending only on $c_0$. The solution of the above linear Beltrami equation can be represented as 
	\[
	F(z) = (h\circ\phi)(z), \  \ z\in B_8,
	\]
	where $\phi : B_8 \rightarrow \phi(B_8)$ is a $K$-quasiconformal mapping with $K$ depending only on $\mu_0$ and $h:\phi(B_8)\to \R^2$ is holomorphic; see e.g.~\cite[Theorem 5.5.1]{Astala:Iwaniec:Martin}.

	\begin{case}
		\it{$1<p<2$:}
	\end{case}
	In this case,
	\[
	\mu := \frac{1- 2^{p-2}A\abs{(Re(F))_z}^{p-2}}{1 + 2^{p-2}A\abs{(Re(F))_z}^{p-2}}.
	\]
	Since $A\in C^{0,\alpha}(\overline{B_8})$, $0<\alpha<1$, the weak solution $v\in C^{1,\beta}(\overline{B_8})$ (cf.~\cite{Lieberman:1988}), i.e., there exists $\beta = \beta(p,\alpha,l)$ with $0<\beta<1$ so that $\|v\|_{C^{1,\beta}(\overline{B_8})} \leq C$ for some positive constant $C = C(p,\alpha,l,\|v\|_{L^\infty(B_8)})$, where $l=\frac{\|A\|_{L^\infty(B_8)}}{c_0}$. In particular, $\|\nabla v\|_{C^{0}(\overline{B_8})} \leq C$ and thus
	\begin{equation}\label{liberm}
	\abs{\nabla v} \leq \|(Re(F))_z\|_{C^{0}(\overline{B_8})} \leq C.
	\end{equation}
	Since we also have $A\geq c_0$ in $B_8$, 
	\[
	\|\mu\|_{L^\infty(B_8)} \leq \frac{1-C_0}{1+C_0} < 1
	\]
	where $C_0 = C_0(p,\alpha,l,\|v\|_{L^\infty(B_8)},c_0)$ is a positive constant.
	
	In either cases, we infer that for some $k<1$,
	\begin{align*}
		\Big|\frac{\partial F}{\partial \bar{z}}\Big|\leq k\Big|\frac{\partial F}{\partial z}\Big|.
	\end{align*}
	This implies that $F$ is a $K$-quasiregular mapping with $K$ depending only on $k$ (and hence only on $p,\alpha$, $l$, $\|v\|_{L^\infty(B_8)}$ and $c_0$) (cf.~\cite[Section 5]{Astala:Iwaniec:Martin}). The final claim follows immediately from the well-known Stoilow factorization; see e.g.~\cite[Corollary 5.5.3]{Astala:Iwaniec:Martin}. 
 
\end{proof}

\begin{proof}[Proof of Theorem \ref{thm:SUCP 1<p<2}]
	We first assume that $z_0=0$ and $B_8\subset \Omega$. Consider the equation~\eqref{eq1}. Let $v\in W^{1,2}(B_8)$ be a weak solution of~\eqref{eq1}. Set $F= v+iw$, where $w$ is the conjugate weighted $q$-harmonic function given as in Lemma \ref{conjugate}. Then $F$ satisfies the nonlinear Beltrami equation
	\begin{equation}\label{pBel}
	\frac{\doo F}{\doo \ol{z}} = \mu\ol{\frac{\doo F}{\doo z}},
	\end{equation}
	where $\mu$ is defined as
	\[
	\mu := \frac{1- 2^{p-2}A\abs{(Re(F))_z}^{p-2}}{1 + 2^{p-2}A\abs{(Re(F))_z}^{p-2}}.
	\]
	Theorem \ref{SolBel1} implies that each $W^{1,2}(B_8)$-solution of \eqref{pBel} can be written as 
	\[
	F(z) = (h\circ\phi)(z) \ \text{in}\  B_8,
	\]
	where $\phi : B_8 \rightarrow \phi(B_8)$ is a $K$-quasiconformal mapping with $K$ depending only on $p,\alpha$, $l$, $\|v\|_{L^\infty(B_8)}$ and $c_0$ as in the previous theorem, and $h:\phi(B_8)\to \R^2$ is holomorphic.
	
	Now we apply the Hadamard's three circle theorem to $F=h\circ\phi$ (cf.~\eqref{eq:three circle} in the proof of Proposition \ref{prop:key prop}) to deduce that
	\begin{equation}\label{weHada}
	\|F\|_{L^{\infty}(B_1)} \leq \|F\|_{L^{\infty}(B_{r/4})}^{\theta} \|F\|_{L^{\infty}(B_2)}^{1-\theta},
	\end{equation}
	where $0<\theta=\theta(r) = \frac{\log\frac{4}{4-(\frac{3}{C_0})^{1/\alpha}}}{\log\frac{4}{(\frac{r}{C_0})^{1/\alpha}}}<1$ with $C_0$ depending only on $K$.

	Applying Lemma \ref{lemma:upper Harnack} for the $K$-quasiregular mapping $F$ and using the H\"older's inequality, we have the following estimates
	\begin{align}
	\|F\|_{L^{\infty}(B_{r/4})}
	& \leq \frac{C}{\abs{B_{r/2}}} \int_{B_{r/2}} \abs{F} \nonumber \\
	& \leq \frac{C}{\abs{B_{r/2}}} \left(\int_{B_{r/2}}1 \right)^{1/p} \left( \int_{B_{r/2}}\abs{F}^{p'}\right)^{1/{p'}} \nonumber \\
	& \leq C \abs{B_{r/2}}^{\frac{1}{p}-1}\|F\|_{L^{p'}(B_r/2)} \nonumber \\
	& \leq C \pi(r/2)^{-2/p'} \|F\|_{L^{p'}(B_{r/2})} \nonumber \\
	& \leq C r^{-2/p'} [\|v\|_{L^{p'}(B_{r/2})} + \|w\|_{L^{p'}(B_{r/2})}], \label{star}
	\end{align}
	where $1/p+1/{p'}=1$ and $C$ depends only on $K$. 
	In order to estimate $\|w\|_{L^{p'}(B_{r/2})}$, we follow similar approach from~\cite[Section 4]{Kenig:Wang:2015} and proceed as follows:
	\[
	\begin{split}
	\int_{B_{r/2}}\abs{w(x)}^{p'}dx
	& = \int_{B_{r/2}}\abs{w(x) - w(0)}^{p'} dx \\
	& = \int_{B_{r/2}} \abs{\int_{0}^{1} \nabla w(tx) \cdot x dt}^{p'} dx \\
	& \leq \left( \frac{r}{2}\right)^{p'} \int_{B_{r/2}} \abs{\int_{0}^{1} \nabla w(tx)}^{p'} dt dx \\
	& \leq \left( \frac{r}{2}\right)^{p'+1} \int_{0}^{r/2} \left\{\frac{1}{\abs{B_s}}\int_{B_s}\abs{\nabla w(y)}^{p'} dy \right\}ds. 
	\end{split}
	\]
	Using the fact that $\abs{\nabla w} = \abs{A}\abs{\nabla v}^{p-1}$, we obtain
	\[
	\begin{split}
	\int_{B_{r/2}}\abs{w(x)}^{p'}dx
	& \leq C r^{p'+1} \int_{0}^{r/2} \left[ \frac{1}{\abs{B_s}} \int_{B_s}\abs{A}^{p'}\abs{\nabla v}^{p'(p-1)} dy\right]ds \\
	& \leq C r^{p'+1} \int_{0}^{r/2} \left[ \frac{1}{\abs{B_s}} \|A\|_{L^{\infty}(B_s)}^{p'}\int_{B_s}\abs{\nabla v}^p dy\right]ds.
	\end{split}
	\]
	Since $\|A\|_{L^{\infty}(B_8)} \leq C_1$, applying Caccioppoli's inequality, Lemma \ref{cacciWE2}, we obtain that
	\begin{equation}\label{mcest}
	\int_{B_{r/2}}\abs{w(x)}^{p'}dx \leq C r^{p'+1}\int_{0}^{r/2}\frac{\|v\|_{L^{\infty}(B_{2s})}^{p}}{s^p\abs{B_s}} ds,
	\end{equation}
	where $C>0$ depends only on $p$, $K$ and $C_1$. From our assumption that $v$ vanishes at $0$ in an infinite order, i.e., $\exists \tilde C>0$ and $\tilde r < 8$ such that
	\begin{equation}\label{starstarstar}
	\abs{v(z)} \leq \tilde C \abs{z}^N, \ \forall \ \abs{z}<\tilde r,
	\end{equation}
	we infer that
	\[
	\int_{B_4}\abs{w(x)}^{p'} \leq C
	\]
	and hence
	\begin{equation}\label{estiii}
	\|F\|_{L^{\infty}(B_2)} \leq C,
	\end{equation}
	where $C>0$ depends on $p$, $K$ and $C_1$. Now, if $\|v\|_{L^{\infty}(B_1)} \geq e^{-l_0}$ for some $l_0>0$, then \eqref{weHada} and \eqref{estiii} would give us
	\[
	\begin{split}
	e^{-l_0} \leq \|v\|_{L^{\infty}(B_1)} 
	& \leq C_1C^{1-\theta} \|F\|_{L^{\infty}(B_{r/4})}^{\theta},
	\end{split}
	\]
	i.e.,
	\[
	\tilde C r^{\tilde C l_0} \leq \|F\|_{L^{\infty}(B_{r/4})}.
	\]
On the other hand, using \eqref{star}, \eqref{mcest} and \eqref{starstarstar}, we may conclude that, there exist $N_0>\tilde{C}l_0$ and $r_{N_0}<\tilde{r}$ so that 
	\[
	\|F\|_{L^{\infty}(B_{r/4})} \leq C_{N_0}r^{N_0}, \ \text{for all}\ r<r_{N_0}.
	\]
	The contradiction implies that for all $l_0>0$,  $\|v\|_{L^{\infty}(B_1)} \leq e^{-l_0}$, and hence $v \equiv 0$ in $B_1$. 
	
	Finally, we will finish the proof using the scaling argument as in \cite{Bourgain:Kenig:2005} and \cite{Kenig:Wang:2015}. Let us take any point $z_0\in \Omega$ so that
	\[
	\abs{v(z)} = \mathcal{O}(\abs{z-z_0}^N) \ \text{as}\ \abs{z-z_0} \rightarrow 0, 
	\]
	for all $N>0$. Choose $r_0>0$ such that $B_{8r_0}(z_0) \subset \Omega$. Set $u(z) = v(z_0+r_0z)$ and $\sigma(z) = r_0A(z_0+r_0z)$. Then $u$ satisfies 
	\[
	\text{div}(\sigma\abs{\nabla u}^{p-2}\nabla u) = 0 \ \text{in}\ B_8
	\]
	and
	\[
	\|\sigma\|_{L^{\infty}(B_8)} \leq C \|A\|_{L^{\infty}(B_{8r_0}(z_0))} \leq C_1.
	\]
	Hence, using the previous argument, we know that $u= 0$ in $B_1$, which implies $v=0$ in $B_{r_0}(z_0).$ Now, by the chain of balls argument as in \cite[Corollary 3.5]{Kenig:Wang:2015}, we easily conclude that $v\equiv 0$ in $\Omega$. 
\end{proof}

\section{Concluding remarks}\label{sec:concluding remark}
In this section, we give a short remark on the local regularity assumption on $W$ in Theorem~\ref{thm:main thm}. We have assumed that $W$ is locally Lipschitz and this assumption was only used in Lemma~\ref{lemma:regularity of F} to ensure that we are legal to take the complex gradient of solutions of~\eqref{eq:main equation} to derive certain quasilinear Beltrami equation. The constant $C$ appeared in the asymptotic estimate in Theorem~\ref{thm:main thm} does not depend on the data associated to the local Lipschitz regularity. Hence, we expect that this assumption is only a technical assumption and can be possibly get rid of by a standard regularizing procedure as follows. Given a general $W$, say belongs to $L^\infty_{loc}(\R^2)$. Let $W_\varepsilon=\eta_\varepsilon*W$ be the standard smooth approximation of $W$ in $L^\infty_{loc}(\R^2)$. For each $\varepsilon>0$, consider the equations
\begin{align}\label{eq:mollifier}
	\dive(|\nabla u|^{p-2}\nabla u)+W_\varepsilon\cdot (|\nabla u|^{p-2}\nabla u)=0\ \text{ in }\R^2.
\end{align}
Applying Theorem~\ref{thm:main thm} for (weak or $C^{1,\alpha}$) solutions $u_\varepsilon$ of \eqref{eq:mollifier}, we may conclude that the quantitative estimates as in Theorem~\ref{thm:main thm} hold for each $u_\varepsilon$ with a constant $C$ independent of $\varepsilon$. It is plausible that the approximation argument from~\cite{DiBenedetto:1983} will imply that $u_\varepsilon\to u$ in $L^\infty_{loc}(\R^2)$. Since $C$ is independent of $\varepsilon$, the same estimate would hold for $u$ as well.

\bibliographystyle{plain}
\bibliography{math11}

\begin{thebibliography}{10}

\bibitem{Alessandrini:1987}
Giovanni Alessandrini.
\newblock Critical points of solutions to the {$p$}-{L}aplace equation in
  dimension two.
\newblock {\em Boll. Un. Mat. Ital. A (7)}, 1(2):239--246, 1987.

\bibitem{Armstrong:Silvestre2011}
Scott~N. Armstrong and Luis Silvestre.
\newblock Unique continuation for fully nonlinear elliptic equations.
\newblock {\em Math. Res. Lett.}, 18(5):921--926, 2011.

\bibitem{Astala:Iwaniec:Martin}
Kari Astala, Tadeusz Iwaniec, and Gaven Martin.
\newblock {\em Elliptic partial differential equations and quasiconformal
  mappings in the plane}, volume~48 of {\em Princeton Mathematical Series}.
\newblock Princeton University Press, Princeton, NJ, 2009.

\bibitem{Astala:Iwaniec:Saksman:2001}
Kari Astala, Tadeusz Iwaniec, and Eero Saksman.
\newblock Beltrami operators in the plane.
\newblock {\em Duke Math. J.}, 107(1):27--56, 2001.

\bibitem{Bojarski:Iwaniec:1987}
B.~Bojarski and T.~Iwaniec.
\newblock {$p$}-harmonic equation and quasiregular mappings.
\newblock In {\em Partial differential equations ({W}arsaw, 1984)}, volume~19
  of {\em Banach Center Publ.}, pages 25--38. PWN, Warsaw, 1987.

\bibitem{bo09}
B.~V. Bojarski.
\newblock {\em Generalized solutions of a system of differential equations of
  the first order and elliptic type with discontinuous coefficients}, volume
  118 of {\em Report. University of Jyv\"askyl\"a Department of Mathematics and
  Statistics}.
\newblock University of Jyv\"askyl\"a, Jyv\"askyl\"a, 2009.
\newblock Translated from the 1957 Russian original, With a foreword by Eero
  Saksman.

\bibitem{Bourgain:Kenig:2005}
Jean Bourgain and Carlos~E. Kenig.
\newblock On localization in the continuous {A}nderson-{B}ernoulli model in
  higher dimension.
\newblock {\em Invent. Math.}, 161(2):389--426, 2005.

\bibitem{Bojarski:1957}
B.~V. Boyarski{\u\i}.
\newblock Generalized solutions of a system of differential equations of first
  order and of elliptic type with discontinuous coefficients.
\newblock {\em Mat. Sb. N.S.}, 43(85):451--503, 1957.

\bibitem{D2014}
Blair Davey.
\newblock Some quantitative unique continuation results for eigenfunctions of
  the magnetic {S}chr\"odinger operator.
\newblock {\em Comm. Partial Differential Equations}, 39(5):876--945, 2014.

\bibitem{DiBenedetto:1983}
E.~{Di Benedetto}.
\newblock {$C^{1+\alpha}$} local regularity of weak solutions of degenerate
  elliptic equations.
\newblock {\em Nonlinear Analysis: Theory, Methods \& Applications}, 7(8):827
  -- 850, 1983.

\bibitem{Garofalo:Lin1986}
Nicola Garofalo and Fang-Hua Lin.
\newblock Monotonicity properties of variational integrals, {$A_p$} weights and
  unique continuation.
\newblock {\em Indiana Univ. Math. J.}, 35(2):245--268, 1986.

\bibitem{Garofalo:Lin1987}
Nicola Garofalo and Fang-Hua Lin.
\newblock Unique continuation for elliptic operators: a geometric-variational
  approach.
\newblock {\em Comm. Pure Appl. Math.}, 40(3):347--366, 1987.

\bibitem{Granlund:Marola2012}
Seppo Granlund and Niko Marola.
\newblock Some remarks on sign changing solutions of a quasilinear elliptic
  equation in two variables.
\newblock {\em arXiv preprint arXiv:1206.1474}, 2012.

\bibitem{Granlund:Marola2014}
Seppo Granlund and Niko Marola.
\newblock On the problem of unique continuation for the {$p$}-{L}aplace
  equation.
\newblock {\em Nonlinear Anal.}, 101:89--97, 2014.

\bibitem{Guo:Kar:Salo:2015}
Chang-Yu Guo, Manas Kar, and Mikko Salo.
\newblock Inverse problems for $p$-laplace type equations under monotonicity
  assumptions.
\newblock {\em Preprint}, 2015.

\bibitem{heinonen01}
Juha Heinonen.
\newblock {\em Lectures on analysis on metric spaces}.
\newblock Universitext. Springer-Verlag, New York, 2001.

\bibitem{Heinonen:Kilpelainen:Martio:1993}
Juha Heinonen, Tero Kilpel{\"a}inen, and Olli Martio.
\newblock {\em Nonlinear potential theory of degenerate elliptic equations}.
\newblock Oxford Mathematical Monographs. The Clarendon Press, Oxford
  University Press, New York, Oxford, 1993.
\newblock Oxford Science Publications.

\bibitem{Kenig:Silvestre:Wang:2015}
Carlos Kenig, Luis Silvestre, and Jenn-Nan Wang.
\newblock On {L}andis' conjecture in the plane.
\newblock {\em Comm. Partial Differential Equations}, 40(4):766--789, 2015.

\bibitem{Kenig:Wang:2015}
Carlos Kenig and Jenn-Nan Wang.
\newblock Quantitative uniqueness estimates for second order elliptic equations
  with unbounded drift.
\newblock {\em Math. Res. Lett.}, 22(4):1159--1175, 2015.

\bibitem{Land1988}
VA~Kondratiev and EM~Landis.
\newblock Qualitative properties of the solutions of a second-order nonlinear
  equation.
\newblock {\em Mat. Sb.(NS)}, 135(177):346--360, 1988.

\bibitem{Lieberman:1988}
Gary~M. Lieberman.
\newblock Boundary regularity for solutions of degenerate elliptic equations.
\newblock {\em Non-Linear Analysis}, 12(11):1203--1219, November 1988.

\bibitem{Lin:Wang:2014}
Ching-Lung Lin and Jenn-Nan Wang.
\newblock Quantitative uniqueness estimates for the general second order
  elliptic equations.
\newblock {\em J. Funct. Anal.}, 266(8):5108--5125, 2014.

\bibitem{Manfredi:1988}
Juan~J. Manfredi.
\newblock {$p$}-harmonic functions in the plane.
\newblock {\em Proc. Amer. Math. Soc.}, 103(2):473--479, 1988.

\bibitem{M1991}
V.~Z. Meshkov.
\newblock On the possible rate of decrease at infinity of the solutions of
  second-order partial differential equations.
\newblock {\em Mat. Sb.}, 182(3):364--383, 1991.

\end{thebibliography}

\end{document}